\documentclass{amsart}
\usepackage{amssymb,amsmath, amsthm,latexsym}
\usepackage{graphics}
\usepackage{psfrag}
\usepackage{amscd}
\usepackage{graphicx}
\newcommand{\cal}[1]{\mathcal{#1}}
\theoremstyle{plain}
\newtheorem{theo}{Theorem}

\newtheorem{lemma}{Lemma}[section]
\newtheorem{theorem}[lemma]{Theorem}  
\newtheorem{proposition}[lemma]{Proposition}
\newtheorem{corollary}[lemma]{Corollary}
\theoremstyle{definition}
\newtheorem{defi}{Definition}
\newtheorem{definition}[lemma]{Definition}
\newtheorem{remark}[lemma]{Remark}
\newtheorem{example}[lemma]{Example}
\let\egthree=\phi
\let\phi=\varphi
\let\varphi=\egthree

\begin{document}
\title[Hyperbolic of relatively hyperbolic graphs]
{Hyperbolic relatively hyperbolic graphs and disc graphs}
\author{Ursula Hamenst\"adt}
\thanks{Partially supported by the Hausdorff Center Bonn
and ERC Grant Nb. 10160104}
\date{May 18, 2014}


\begin{abstract}
We show that a relatively hyperbolic graph with
uniformly hyperbolic peripheral subgraphs is hyperbolic.
As an application, we show that the disc graph and
the electrified disc graph of 
a handlebody $H$ of genus $g\geq 2$
are hyperbolic, and we determine their Gromov boundaries.
\end{abstract}

\maketitle


\section{Introduction}

Consider a
connected metric graph ${\cal G}$ in which a family
${\cal H}=\{H_c\mid c\in {\cal C}\}$ of complete
connected subgraphs has been
specified. Here ${\cal C}$ is a 
countable, finite or empty index set. 
The graph ${\cal G}$ is \emph{hyperbolic
relative to the family ${\cal H}$} if 
the following properties are satisfied.

Define 
the \emph{${\cal H}$-electrification} ${\cal E\cal G}$ 
of ${\cal G}$ to be the graph which is obtained
from ${\cal G}$ by adding for every $c\in {\cal C}$ a
new vertex $v_c$ which is connected to each 
vertex $x\in H_c$ by an edge and which is not connected
to any other vertex. We require that 
the graph ${\cal E\cal G}$ is hyperbolic in the sense of Gromov
and that moreover
a property called \emph{bounded penetration} holds true
(see \cite{F98} for perhaps the first formulation of this
property). We refer
to \cite{S12} for a consolidation of the various
notions of relative hyperbolicity found in the literature.

If ${\cal G}$ is a hyperbolic metric 
graph and if 
${\cal H}$ is a family of disjoint connected uniformly quasi-convex
subgraphs of ${\cal G}$ then ${\cal G}$ is hyperbolic relative
to ${\cal H}$. This fact is probably folclore; implicitly it
was worked out in a slightly modified form in 
\cite{KR12}.  

Vice versa, Farb showed in \cite{F98} 
that if ${\cal G}$ is the
Cayley graph of a finitely generated group and if 
the graphs $H_c$ are $\delta$-hyperbolic 
for a number $\delta >0$ not depending on $c\in {\cal C}$ then
${\cal G}$ is hyperbolic. 
In \cite{BF06} it is noted that using a result of 
Bowditch \cite{Bw91}, the argument in \cite{F98} 
can be extended to arbitrary (possibly locally infinite) 
relatively hyperbolic metric graphs.

Our first goal is to give a different and self-contained proof of this
result which gives effective
estimates for the hyperbolicity constant as well as
explicit control on uniform quasi-geodesics. We show

\begin{theo}\label{hypback}
Let ${\cal G}$ be a metric graph which is hyperbolic
relative to a family ${\cal H}=\{H_c\mid c\in {\cal C}\}$
of complete connected subgraphs. If there is a number
$\delta>0$ such that each of the graphs $H_c$ is 
$\delta$-hyperbolic then ${\cal G}$ is hyperbolic. 
Moreover, the subgraphs $H_c$ $(c\in {\cal C})$ are uniformly
quasi-convex.
\end{theo}

The control we obtain allows to use the result 
inductively. Moreover, the Gromov boundary 
of ${\cal G}$ can easily be determined from the 
Gromov boundaries of ${\cal E\cal D\cal G}$ and the
Gromov boundaries of the quasi-convex subgraphs $H_c$.

We next discuss applications of Theorem \ref{hypback}.

Let $S$ be a closed 
surface of genus $g\geq 2$. 
For a number $k<g$ define the 
\emph{graph of non-separating $k$-multicurves} to be the
following metric graph ${\cal N\cal C}(k)$. Vertices 
are $k$-tuples of simple closed curves on $S$ which 
cut $S$ into a single connected component. 
Two such multicurves $c_1,c_2$ are connected by an edge
if $c_1\cup c_2$ is a non-separating multicurve with 
$k+1$ components.  In \cite{H13b} we used 
Theorem \ref{hypback} to show

\begin{theo}\label{nonsepmult}
For $k<g/2+1$ the graph ${\cal N\cal C}(k)$ is hyperbolic.
\end{theo}

We also observed that the bound $k<g/2+1$ is sharp.
The same argument applies to the graph of 
non-separating multi-curves on a surface with punctures.

In this article we use Theorem \ref{hypback} 
to investigate the geometry
of graphs of discs in a handlebody. 
A handlebody of genus $g\geq 1$  
is a compact three-dimensional manifold $H$  which can
be realized as a closed regular neighborhood in $\mathbb{R}^3$
of an embedded bouquet of $g$ circles. Its boundary
$\partial H$ is an oriented surface of genus $g$.

An \emph{essential disc} in $H$ is a properly embedded
disc $(D,\partial D)\subset (H,\partial H)$ whose
boundary $\partial D$
is an essential simple closed curve in $\partial H$.

A subsurface $X$ of the compact surface $\partial H$ is 
called \emph{essential} if it is a complementary 
component of an embedded multicurve in $\partial H$.
Note that the complement of a non-separating simple closed
curve in $\partial H$ is essential in this sense, i.e. 
the inclusion $X\to \partial H$ need not induce in injection 
on fundamental groups.

Define a connected essential subsurface
$X$ of the boundary $\partial H$ of $H$ to be \emph{thick} if 
the following properties hold true.
\begin{enumerate}
\item Every disc intersects $X$.
\item $X$ is filled by boundaries of discs.
\end{enumerate}

The boundary surface $\partial H$ of $H$ is thick. 
An example of a proper thick subsurface of $\partial H$ is 
the complement in $\partial H$ of a suitably chosen simple closed curve
which is not discbounding.

\begin{defi}\label{defineelec}
Let $X\subset \partial H$ be a thick subsurface. 
The \emph{electrified disc graph} of $X$ 
is the graph 
${\cal E\cal D\cal G}(X)$ whose vertices
are isotopy classes of essential
discs in $H$ with boundary in $X$.
Two vertices $D_1,D_2$ are connected by an edge of length one
if there is an essential simple closed curve
in $X$ which can be realized disjointly 
from both $\partial D_1,\partial D_2$.
\end{defi}

If $X=\partial H$ then we call ${\cal E\cal D\cal G}(X)$ the 
\emph{electrified disc graph} of $H$.
Using Theorem \ref{hypback} we show

\begin{theo}\label{electrified}
The electrified disc graph ${\cal E\cal D\cal G}(X)$ of a thick subsurface
$X\subset \partial H$ of the boundary 
$\partial H$ of a handlebody $H$ of genus
$g\geq 2$ is hyperbolic.
\end{theo}

For the investigation of the \emph{handlebody group}, i.e. the group of 
isotopy classes of homeomorphisms of $H$, a more natural graph
to consider   
is the so-called disc graph 
which is defined as follows.

\begin{defi}\label{disccomplex}
The \emph{disc graph} ${\cal D\cal G}$
of $H$ is the graph whose vertices
are isotopy classes of essential discs in $H$. 
Two such discs are connected by an edge 
of length one if and only if
they can be realized disjointly.
\end{defi}

Since 
for any two disjoint  essential 
simple closed curves $c,d$ on $\partial H$
there is a simple closed curve on $\partial H$
which can be realized disjointly from $c,d$
(e.g. one of the curves $c,d$),
the electrified disc graph is obtained from the disc graph by adding 
some edges. This observation allows to apply Theorem \ref{hypback}
inductively to the graphs ${\cal E\cal D\cal G}(X)$ where
$X$ passes through the thick subsurfaces of $\partial H$ 
and deduce in a bottom-up 
inductive procedure hyperbolicity
of the disc graph from hyperbolicity of the 
electrified disc graph. In this way we obtain
a new, completely combinatorial and 
significantly simpler proof of
the following result which was first established by 
Masur and Schleimer \cite{MS13}.

\begin{theo}\label{discgraphs}
The disc graph ${\cal D\cal G}$ of 
a handlebody $H$ of genus $g\geq 2$ is hyperbolic.
\end{theo}

We also determine the Gromov boundary of the disc graph.
Namely, recall from \cite{K99,H06} 
that the Gromov boundary of the curve graph 
of an essential subsurface $X$ of $\partial H$ can
be identified with the space of minimal geodesic laminations
$\lambda$ in $X$ which \emph{fill} $X$, i.e. are such that every 
essential simple closed curve in $X$ has non-trivial intersection 
with $\lambda$. The Gromov topology on this space of 
geodesic laminations is the \emph{coarse Hausdorff
topology} which can be defined as follows. 
A sequence $\lambda_i$ converges to $\lambda$ if and only if
every limit in the usual Hausdorff topology 
of a subsequence of $\lambda_i$ contains $\lambda$ as a 
sublamination. Notice that 
the coarse Hausdorff topology is defined on the entire space
${\cal L}(\partial H)$ 
of geodesic laminations on $\partial H$, however it is not 
Hausdorff.

We observe that for every thick subsurface $X$ of 
$\partial H$ the Gromov boundary 
$\partial {\cal E\cal D\cal G}(X)$ of 
the electrified disc graph ${\cal E\cal D\cal G}(X)$ 
can be identified with a subspace of the space of topological 
laminations on $X$, equipped with the coarse Hausdorff topology.
Moreover we show

\begin{theo}\label{gromovbd} The Gromov boundary 
$\partial {\cal D\cal G}$ of the disc graph equals the
subspace 
\[\partial {\cal D\cal G}=\cup_X \partial {\cal E\cal D\cal G}(X)
\subset {\cal L}(\partial H)\]
equipped with the coarse Hausdorff topology. The union is over
all thick subsurfaces $X$ of $\partial H$.
\end{theo}

There is no analog of this result for handlebodies with 
\emph{spots}, i.e. with marked points on the boundary.
Indeed, we showed in \cite{H13} that the disc graph of 
a handlebody with one or two spots on the boundary is not 
hyperbolic. The electrified disc graph is not hyperbolic 
for handlebodies with one spot on the boundary, and the same
holds true for sphere graphs.

The organization of this paper is as follows. 
In Section 2 we show Theorem \ref{hypback}.
Section 3 discusses some relative version of 
results from \cite{H11}. In Section 4, we show the
second part of Theorem \ref{discgraphs}, and 
the proof of the first part as well 
as of Theorem \ref{gromovbd} is contained in Section 5.

\section{Hyperbolic thinnings of hyperbolic graphs}

In this section we show Theorem \ref{hypback}
from the introduction.
Consider a (not necessarily locally finite) metric graph 
${\cal G}$ (i.e. edges have length one) 
and a family ${\cal H}=\{H_c\mid c\in {\cal C}\}$
of complete connected subgraphs, where ${\cal C}$ is 
any countable, finite or empty index set.

Define the \emph{${\cal H}$-electrification}
of ${\cal G}$ to be the 
metric graph $({\cal E\cal G},d_{\cal E})$ which is obtained
from ${\cal G}$ by adding vertices and edges as follows.
For each $c\in {\cal C}$ there is a unique vertex 
$v_c\in {\cal E\cal G}-{\cal G}$. This vertex
is connected with each of the vertices of $H_c$ by a single
edge of length one, and it is not connected with any other vertex.

In the sequel all parametrized 
paths $\gamma$ in ${\cal G}$ or ${\cal E\cal G}$
are supposed to be \emph{simplicial}. This means that
the image of every integer is a vertex, and the image
of an integral interval $[k,k+1]$ is an edge or a single vertex.

Call a simplicial path $\gamma$ in ${\cal E\cal G}$ 
\emph{efficient} if for every $c\in {\cal C}$
we have $\gamma(k)=v_c$ for at most one $k$.
Note that if $\gamma$ is an efficient 
simplicial path in ${\cal E\cal G}$
which passes through $\gamma(k)=v_c$ for some $c\in {\cal C}$
then $\gamma(k-1)\in H_c,\gamma(k+1)\in H_c$.

The following definition is an adaptation of a definition 
from \cite{F98}.

\begin{definition}\label{bcp}
The family ${\cal H}$ has 
the \emph{bounded penetration property} if 
for every $L>0$ there is a number
$p(L)>2r$ with the following property.
Let $\gamma$ be an efficient $L$-quasi-geodesic in 
${\cal E\cal G}$, let $c\in {\cal C}$ and 
let $k\in \mathbb{Z}$ be such that
$\gamma(k)=v_c$. If the distance in $H_c$ between
$\gamma(k-1)$ and $\gamma(k+1)$ is at least $p(L)$ then
every efficient $L$-quasi-geodesic $\gamma^\prime$ 
in ${\cal E\cal G}$ with the
same endpoints as $\gamma$ passes through $v_c$. 
Moreover, if $k^\prime\in \mathbb{Z}$ is such that
$\gamma^\prime(k^\prime)=v_c$ then
the distance in $H_c$ between 
$\gamma(k-1),\gamma^\prime(k^\prime-1)$ and between
$\gamma(k+1),\gamma^\prime(k^\prime+1)$ is at most $p(L)$.
\end{definition}

The definition of relative hyperbolicity for a graph 
below is taken from \cite{S12} where it is shown
to be equivalent to other definitions of
relative hyperbolicity found in the literature.

\begin{definition}\label{realtivhyp}
Let ${\cal H}$ be a family of complete connected
subgraphs of a metric graph ${\cal G}$.
The graph ${\cal G}$ is \emph{hyperbolic relative to 
${\cal H}$} if the ${\cal H}$-electrification of 
${\cal G}$ is hyperbolic and if moreover 
${\cal H}$ has the bounded penetration property.
\end{definition}

From now on we always consider a metric graph
${\cal G}$ which is hyperbolic relative
to a family ${\cal H}=\{H_c\mid c\in {\cal C}\}$ 
of complete connected subgraphs.

We say that the family ${\cal H}$ is \emph{$r$-bounded} 
for a number $r>0$ 
if ${\rm diam}(H_c\cap H_d)
\leq r$ for $c\not=d\in {\cal C}$ where the diameter
is taken with respect to the intrinsic path metric on $H_c$ and
$H_d$. A family which is $r$-bounded for some $r>0$ is
simply called bounded. 

The following is a consequence of the main theorem of 
\cite{S12} (the equivalence of definition RH0 and RH2).

\begin{proposition}
If ${\cal G}$ is hyperbolic relative to the family
${\cal H}$ then ${\cal H}$ is bounded.
\end{proposition}

Let ${\cal H}$ be as in Definition \ref{bcp}.
Define an \emph{enlargement} $\hat \gamma$ of an efficient
 simplicial $L$-quasi-geodesic $\gamma:[0,n]\to {\cal E\cal G}$ 
with endpoints $\gamma(0),\gamma(n)\in {\cal G}$ as
follows. Let $0<k_1<\dots <k_s< n$ be those points
such that $\gamma(k_i)=v_{c_i}$ for some $c_i\in {\cal C}$.
Then $\gamma(k_i-1),\gamma(k_i+1)\in H_{c_i}$. 
For each $i\leq s$ replace 
$\gamma[k_i-1,k_i+1]$ by a simplicial geodesic in $H_{c_i}$ with
the same endpoints.

For a number $k>0$ define a subset $Z$ of the 
metric graph ${\cal G}$ to be 
\emph{$k$-quasi-convex}
if any geodesic with both endpoints in $Z$ is contained in the
$k$-neighborhood of $Z$. In particular, up to perhaps increasing
the number $k$, any two points in $Z$ can be connected
in $Z$ by a (not necessarily continuous) path which is 
a $k$-quasi-geodesic in ${\cal G}$.
The goal of this section is to show

\begin{theorem}\label{hypextension}
Let ${\cal G}$ be a metric graph which is 
hyperbolic relative to a family 
${\cal H}=\{H_c\mid c\}$ of complete connected subgraphs.
If there is a 
number $\delta >0$ such that 
each of the graphs $H_c$ is $\delta$-hyperbolic
then ${\cal G}$ is hyperbolic. Enlargements of 
geodesics in ${\cal E\cal G}$ are uniform quasi-geodesics in ${\cal G}$.
The subgraphs $H_c$ are uniformly quasi-convex.
\end{theorem}

For the remainder of this section we assume that
${\cal G}$ is a graph which is hyperbolic
relative to a family ${\cal H}$ of complete connected  
$\delta$-hyperbolic subgraphs.

For a number $R>2r$ 
call $c\in {\cal C}$ \emph{$R$-wide} for  
an efficient 
$L$-quasi-geodesic $\gamma$ in ${\cal E\cal G}$
if the following holds true. There is some 
$k\in \mathbb{Z}$ such that 
$\gamma(k)=v_c$, and
the distance between $\gamma(k-1),\gamma(k+1)$ in $H_c$ 
is at least $R$. Note that since ${\cal H}$ is $r$-bounded, 
$c$ is uniquely determined
by $\gamma(k-1),\gamma(k+1)$. 
If $R=p(L)$ is as in Definition \ref{bcp} then
we simply say that $c$ is \emph{wide}.

\begin{lemma}\label{wideiswide}
Let $L\geq 1$ and let $\gamma_1,\gamma_2$ be two 
efficient $L$-quasi-geodesics
in ${\cal E\cal C}$ with the same endpoints. If 
$c\in {\cal C}$ is $3p(L)$-wide for $\gamma_1$ then 
$c$ is wide for $\gamma_2$.
\end{lemma}
\begin{proof} By definition, if $c$ is $3p(L)$-wide for 
$\gamma_1$ then there is some $k$ so that 
$\gamma_1(k)=v_c$ and that the distance in $H_c$ between
$\gamma_1(k-1)$ and $\gamma_1(k+1)$ is at least $3p(L)$.
Since $\gamma_2$ is an efficient $L$-quasi-geodesic with the same
endpoints as $\gamma_1$, 
by the bounded penetration property
there is some $k^\prime$ so that
$\gamma_2(k^\prime)=v_c$, moreover the distance
in $H_c$ between $\gamma_1(k-1)$ and $\gamma_2(k^\prime-1)$ and
between $\gamma_1(k+1)$ and $\gamma_2(k^\prime+1)$ is 
at most $p(L)$. Thus by the triangle inequality, 
the distance in $H_c$ between 
$\gamma_2(k^\prime-1)$ and $\gamma_2(k^\prime+1)$
is at least $p(L)$ which is what we wanted to show.
\end{proof}

Define the \emph{Hausdorff distance} between two 
closed subsets $A,B$ of a metric space 
to be the infimum of the numbers $b>0$ such that
$A$ is contained in the $b$-neighborhood of $B$ and 
$B$ is contained in the $b$-neighborhood of $A$.

\begin{lemma}\label{Lipschitz}
For every $L>0$ there is a 
number $\kappa(L)>0$ with the following property.
Let $\gamma_1,\gamma_2$ be two 
efficient simplicial $L$-quasi-geodesics
in ${\cal E\cal G}$ connecting the same points
in ${\cal G}$, with
enlargements $\hat\gamma_1,\hat\gamma_2$. Then
the Hausdorff distance in ${\cal G}$
between the images of $\hat\gamma_1,
\hat\gamma_2$ is at most $\kappa(L)$.
\end{lemma} 
\begin{proof}
Let $\gamma:[0,n]\to {\cal E\cal G}$ be an efficient
simplicial $L$-quasi-geodesic with endpoints
$\gamma(0),\gamma(n)\in {\cal G}$. Let $R>p(L)$ and 
assume that  
$c\in {\cal C}$ is
not $R$-wide for $\gamma$.
If there is some $u\in \{1,\dots,n-1\}$ such that $\gamma(u)=v_c$ then 
$\gamma(u-1),\gamma(u+1)\in H_c$.
Since $c$ is not $R$-wide for $\gamma$,
$\gamma(u-1)$ can
be connected to $\gamma(u+1)$ by an arc in 
$H_c$ of length at most $R$. 
In particular, if no $c\in {\cal C}$ 
is $R$-wide for $\gamma$
then an enlargement $\hat \gamma$ of $\gamma$
is an $\hat L$-quasi-geodesic in 
${\cal E\cal G}$ for a universal constant $\hat L=
\hat L(L,R)>0$. 
Then $\hat \gamma$ is a $\hat L$-quasi-geodesic
in ${\cal G}$ as well (note that the inclusion 
${\cal G}\to {\cal E\cal G}$ is $1$-Lipschitz).

Let $\gamma_i:[0,n_i]\to 
{\cal E\cal G}$ be efficient $L$-quasi-geodesics
$(i=1,2)$ with the same endpoints in ${\cal G}$.
Assume that no $c\in {\cal C}$ 
is wide for $\gamma_1$. By Lemma \ref{wideiswide},  
no $c\in {\cal C}$ is $R=3p(L)$-wide for $\gamma_2$.
Let $\hat\gamma_i$ be an enlargement of $\gamma_i$.
By the above discussion, the arcs
$\hat \gamma_i$
are $\hat L=\hat L(L,3p(L))$-quasi-geodesics in ${\cal E\cal G}$.
In particular, by hyperbolicity of ${\cal E\cal G}$, 
the Hausdorff  distance 
in ${\cal E\cal G}$ 
between the images of $\hat\gamma_i$ 
is bounded from above
by a constant $b-1>0$ only depending on $L$ and $R$.

We have to show that the Hausdorff distance in 
${\cal G}$ between these images 
is also uniformly bounded. 
For this let 
$x=\hat \gamma_1(u)$ be any vertex on $\hat\gamma_1$
and let $y=\hat\gamma_2(w)$ be a vertex on $\hat\gamma_2$
of minimal distance in ${\cal E\cal G}$ to $x$. Then 
$d_{\cal E}(x,y)\leq b$ (here as before, $d_{\cal E}$ is the 
distance in ${\cal E\cal G}$, and we let
$d$ be the distance in ${\cal G}$).
Let $\zeta$ be a geodesic in ${\cal E\cal G}$ connecting
$x$ to $y$. Since $y$ is a vertex on $\hat\gamma_2$ of minimal
distance to $x$, $\zeta$ intersects $\hat\gamma_2$ only at its endpoints. 

We claim that there is a universal constant
$\chi >0$ such that no $c\in {\cal C}$  
is $\chi$-wide for $\zeta$. 
Namely, since $\hat\gamma_1$ does not pass through any
of the special vertices in ${\cal E\cal G}$, 
 the concatenation $\xi=\zeta\circ \hat \gamma_1[0,u]$
is efficient. Thus $\xi$ is an efficient
$L^\prime$-quasi-geodesic
in ${\cal E\cal G}$ 
with the same endpoints as $\hat\gamma_2[0,w]$ where
$L^\prime>\hat L>L$ only depends on $L$.
Hence by the bounded penetration property,  
if $c\in {\cal C}$ is 
$p(L^\prime)$-wide for $\zeta$ then
the 
$\hat L$-quasi-geodesic $\hat\gamma_2[0,w]$ 
passes through the vertex $v_c$ which is 
a contradiction.

As a consequence of the above
discussion, the length of an enlargement of $\zeta$ 
is bounded from above by a fixed multiple of
$d_{\cal E}(\hat\gamma_1(u),\hat\gamma_2(w))$,
i.e. it is uniformly bounded.
This shows that $d(\hat\gamma_1(u),\hat\gamma_2(w))$ is uniformly
bounded. As a consequence, the image of $\hat \gamma_1$
is contained in a neighborhood of uniformly bounded diameter
in ${\cal G}$ of the image of $\hat\gamma_2$. Exchanging
$\gamma_1$ and $\gamma_2$ we conclude that indeed
the Hausdorff distance in ${\cal G}$ between
the images of the enlargements $\hat\gamma_1,\hat\gamma_2$
is bounded by a number only depending on $L$. 

Now let $\gamma_j:[0,n_j]\to 
{\cal E\cal G}$ be arbitrary efficient 
$L$-quasi-geodesics $(j=1,2)$ connecting the  
same points in ${\cal G}$. 
Then there are numbers $0< u_1<\dots <u_k<n_1$
such that for every $i\leq k$, $\gamma_1(u_i)=v_{c_i}$
where $c_i\in {\cal C}$ 
is wide for $\gamma_1$, and there are no other
wide points for $\gamma_1$. Put $u_0=-1$
and $u_{k+1}=n_1+1$.

By the bounded penetration property, there are numbers 
$w_i\in \{1,\dots,n_2-1\}$ such that
$\gamma_2(w_i)=\gamma_1(u_i)=v_{c_i}$ for all $i$.
Moreover, the distance in $H_{c_i}$ between
$\gamma_1(u_{i}-1)$ and $\gamma_2(w_{i}-1)$ and
between $\gamma_1(u_{i}+1)$ and 
$\gamma_2(w_{i}+1)$ is at most $p(L)$. 
Since $\gamma_1,\gamma_2$ are $L$-quasi-geodesics
by assumption, we may assume that the special vertices
$v_{c_i}$
appear along $\gamma_2$ in the same order 
as along $\gamma_1$, i.e. that
$0<w_1<\dots <w_k<n_2$.
Put $w_0=-1$ and $w_{k+1}=n_2+1$.

For each $i\leq k$, 
define a simplicial edge path 
$\zeta_i:[a_i,a_{i+1}]\to {\cal E\cal G}$
connecting 
$\zeta_i(a_i)=\gamma_1(u_i+1)\in H_{c_i}$ 
to $\zeta_i(a_{i+1})=\gamma_1(u_{i+1}-1)\in H_{c_{i+1}}$
as the concatentation of the following three arcs.
A geodesic in $H_{c_i}$ connecting
$\gamma_1(u_i+1)$ to $\gamma_2(w_i+1)$
(whose length is at most $p(L)$), the arc 
$\gamma_2[w_i+1,w_{i+1}-1]$ and a geodesic
in $H_{c_{i+1}}$ connecting
$\gamma_2(w_{i+1}-1)$ to $\gamma_2(u_{i+1}-1)$.  
Let moreover $\eta_i=\gamma_1\vert [u_i+1,u_{i+1}-1]$
$(i\geq 0)$.
Then $\eta_i,\zeta_i$ are efficient uniform quasi-geodesics
in ${\cal E\cal G}$ with the same endpoints, and 
$\eta_i$ does not have wide points.

For each $i$ let $\hat\nu_i$ be an enlargement of the arc 
$\nu_i=\gamma_2[w_i+1,w_{i+1}-1]$.
By construction, there is an enlargement $\hat\zeta_i$ 
of the efficient 
quasi-geodesic $\zeta_i$ 
which contains $\hat \nu_i$ as a subarc and 
whose Hausdorff distance in ${\cal G}$ 
to $\hat\nu_i$ is uniformly bounded. Let $\hat \eta_i$
be an enlargement of $\eta_i$. Then 
$\hat\zeta_i,\hat\eta_i$ are enlargements of the efficient 
uniform quasi-geodesics $\zeta_i,\eta_i$ 
in ${\cal E\cal G}$ 
with the same endpoints, and 
$\eta_i$ does not have wide points. Therefore 
by the first part of this proof, the Hausdorff distance
in ${\cal G}$ between
$\hat \eta_i$ and $\hat\zeta_i$ 
is uniformly bounded. 
 Hence the Hausdorff
distance between $\hat\eta_i$ and $\hat\nu_i$ is 
uniformly bounded as well.

There is an enlargement $\hat \gamma_1$ 
of $\gamma_1$ which can be represented as 
\[\hat\gamma_1=\hat\eta_k\circ \sigma_k\circ \dots 
\circ \sigma_1\circ \hat \eta_0\]
where for each $i$, $\sigma_i$ is a geodesic in 
$H_{c_i}$ connecting $\gamma_1(u_i-1)$ to 
$\gamma_1(u_i+1)$.   
Similarly, there is an enlargement
$\hat\gamma_2$ of $\gamma_2$ which can be represented as
\[\hat\gamma_2=\hat\nu_k\circ \tau_k\circ \dots 
\circ \tau_1\circ \hat \nu_0\]
where for each $i$, $\tau_i$ is a geodesic
in $H_{c_i}$ connecting $\gamma_2(w_i-1)$ to 
$\gamma_2(w_i+1)$. 
 
For each $i$ 
the distance in $H_{c_i}$ 
between $\gamma_1(u_i-1)$ and 
$\gamma_2(w_i-1)$ is at most $p(L)$, and the same
holds true for the distance between 
$\gamma_1(u_i+1)$ and $\gamma_2(w_i+1)$. Since 
$H_{c_i}$ 
is $\delta$-hyperbolic
for a constant $\delta>0$ not depending on $c_i$, the Hausdorff
distance in $H_{c_i}$ between
any two geodesics 
connecting $\gamma_1(u_i-1)$
to $\gamma_1(u_i+1)$ 
and connecting 
$\gamma_2(w_i-1)$ to $\gamma_2(w _i+1)$ is uniformly bounded.
Together with the above discussion,
this shows the lemma.
\end{proof}

Let for the moment $X$ be an arbitrary geodesic
metric space. Assume that  
for every pair of points $x,y\in X$ there is a fixed
choice of a path $\rho_{x,y}$
connecting $x$ to $y$. 
The \emph{thin triangle property} for this family
of paths states that
there is a universal number $C>0$ so that for 
any triple $x,y,z$ of points in $X$, the image
of $\rho_{x,y}$ is contained in the $C$-neighborhood of the
union of the images of $\rho_{y,z},\rho_{z,x}$. 

For two vertices $x,y\in {\cal G}$ let $\rho_{x,y}$
be an enlargement of a geodesic in ${\cal E\cal G}$
connecting $x$ to $y$. We have

\begin{proposition}\label{thin}
The thin triangle inequality property holds true for the paths
$\rho_{x,y}$.
\end{proposition}
\begin{proof} Let $x_1,x_2,x_3$ be three vertices in ${\cal G}$
and for $i=1,2,3$ let $\gamma_i:[0,n_i]\to {\cal E\cal G}$
be a geodesic connecting $x_i$ to $x_{i+1}$.

By hyperbolicity of ${\cal E\cal G}$ there is a 
number $L>0$ not depending on the points $x_i$, and there is  
a vertex
$y\in {\cal E\cal G}$ with the following property. For $i=1,2,3$ let
$\beta_i:[0,p_i]\to {\cal E\cal G}$ 
be a geodesic in ${\cal E\cal G}$ connecting
$x_i$ to $y$. Then for all $i$, $\alpha_i=\beta_{i+1}^{-1}\circ \beta_i$ is an 
$L$-quasi-geodesic connecting $x_i$ to $x_{i+1}$.

We claim that without loss of generality
we may assume that the quasi-geodesics
$\alpha_i$ are efficient. Namely, 
since the arcs $\beta_i$ are geodesics, they do not backtrack. 
Thus if $\alpha_1$ is \emph{not} efficient then
there is a common point $y$ on $\beta_1$ and $\beta_2$.
Let $s_1<p_1$ be the smallest number so that
$\beta_1(s_1)=\beta_2(s_2)$ for some
$s_2\in \beta_2[0,p_2]$. Then the distance between
$y$ and $\beta_i(s_i)$ $(i=1,2)$ is uniformly bounded,
and $\tilde\alpha_1=(\beta_2 [0,s_2])^{-1}\circ 
\beta_1[0,s_1]$ is an efficient $L$-quasi-geodesic 
connecting $x_1$ to $x_2$. 
Replace $y$ by $\beta_1(s_1)$, 
 replace $\beta_i$
by $\tilde \beta_i=\beta_i[0,s_i]$ $(i=1,2)$ and 
and replace $\beta_3$ by
a geodesic $\tilde \beta_3:[0,\tilde p_3]\to {\cal E\cal G}$
connecting $x_3$ to $\beta_1(s_1)$. 
Thus 
up to increasing the number $L$ by a uniformly bounded
amount 
we may assume that the quasi-geodesic  $\alpha_1$ is efficient.

Assume from now on that $\beta_1,\beta_2,\beta_3$ are such that
the quasi-geodesic $\alpha_1=\beta_2^{-1}\circ \beta_1$ is efficient. 
Using the notation from the second paragraph of this proof, 
if there is some $s<p_3$ such that $\beta_3(s)$ 
is contained in
$\hat\alpha_1$ then let $s_3$ be the smallest number
with this property. 
Replace the point $y=\beta_i(p_i)$ by $\beta_3(s_3)$,
replace $\beta_3$ by $\beta_3[0,s_3]$ and for $i=1,2$ replace
$\beta_i$ by the subarc of $\alpha_1$ connecting
$x_i$ to $\beta_3(s_3)$. With this construction, up to
increasing the number $L$ by a uniformly bounded amount
and perhaps replacing $\beta_1,\beta_2$ by uniform quasi-geodesics
rather than geodesics we may assume that all 
three quasi-geodesics $\tilde \alpha_i=
\tilde \beta_{i+1}^{-1}\circ \tilde \beta_i$ 
$(i=1,2,3)$ are efficient.

Resuming notation, assume from now on that the 
quasi-geodesics $\alpha_i$ are efficient. 
By Lemma \ref{Lipschitz}, the Hausdorff distance
between an enlargement of the geodesic $\gamma_i$ and 
any choice of an 
enlargement of the efficient uniform quasi-geodesic $\alpha_i$ with
the same endpoints is uniformly bounded.
Thus it suffices to show the thin triangle 
property for enlargements of the quasi-geodesics $\alpha_i$.

If $y\in {\cal G}$ then an enlargement
of the quasi-geodesic $\alpha_i$ is the concatenation of 
an enlargement of the quasi-geodesic 
$\beta_i$ and
an enlargement of the quasi-geodesic 
$\beta_{i+1}^{-1}$ which have endpoints in ${\cal G}$. 
Hence in this case the thin triangle
property follows once more from Lemma \ref{Lipschitz}.

If $y=v_c$ for some $c\in {\cal C}$ then 
we distinguish two cases. 

{\sl Case 1:} $c\in {\cal C}$ is wide for each of the
quasi-geodesics $\alpha_i$. 

Recall that $y=\beta_i(p_i)$. 
By hyperbolicity of $H_c$,
there is a number $R>0$ not depending on $c$
such that for all $i\in \{1,2,3\}$ the image of 
any geodesic in $H_c$ connecting
$\beta_i(p_i-1)$ to $\beta_{i+1}(p_{i+1}-1)$ 
 is contained in the $R$-neighborhood of the 
 union of the images of any two geodesics 
connecting $\beta_j(p_j-1)$ to 
$\beta_{j+1}(p_{j+1}-1)$ for $j\not=i$ and where
indices are taken modulo three.
In other words, the thin triangle property holds true for such 
geodesics.

Now let $\hat\alpha_i$ be an enlargement of $\alpha_i$
and let $\zeta_i$ be the subarc of $\hat\alpha_i$ 
which connects $\beta_i(p_i-1)$ to $\beta_{i+1}(p_{i+1}-1)$.
By the definition of an enlargement, $\zeta_i$ is a geodesic
in $H_c$. Thus by the discussion in the previous 
paragraph and by the fact that we may use the same
enlargement of the arc $\beta_{i+1}[0,p_{i+1}-1]$ for the 
construction of an enlargement of $\alpha_{i}$ and $\alpha_{i+1}$, 
the thin triangle property holds true for some
suitable choice and hence any choice of an enlargement of 
the quasi-geodesics 
$\alpha_i$ which is what we wanted to show.

{\sl Case 2:} For at least one $i$, $c\in {\cal C}$ is not
wide for the quasi-geodesic $\alpha_i$.

Assume that this holds true for the quasi-geodesic $\alpha_1$.
Then the distance in $H_c$ between $\beta_1(p_1-1)$ and 
$\beta_2(p_2-1)$ is uniformly bounded (depending on the 
quasi-geodesic constant for $\alpha_1$).
Replace the point $y$ by $\beta_1(p_1-1)$, replace the
quasi-geodesic $\beta_1$ by $\tilde \beta_1=\beta_1[0,p_1-1]$, 
replace the quasi-geodesic $\beta_2$ by the concatentation 
$\tilde \beta_2$ of $\beta_2[0,p_2-1]$ with a geodesic in 
$H_c$ connecting $\beta_2(p_2-1)$ to $\beta_1(p_1-1)$,
and replace the geodesic 
$\beta_3$ by the concatentation $\tilde \beta_3$ of $\beta_3$ with the
edge connecting $v_c$ to $\beta_1(p_1-1)$.
The resulting arcs $\tilde \beta_i$ are efficient uniform quasi-geodesics in 
${\cal E\cal G}$, and they connect the points $x_i$ to $y\in {\cal G}$.
Moreover, the quasi-geodesics $\tilde \beta_{i+1}^{-1}\circ \tilde \beta_i$ 
are effficient as well and hence we are done by the 
above proof for the case $y\in {\cal G}$. 
\end{proof}

Now we are ready to show

\begin{corollary}\label{secondlevel}
${\cal G}$ is hyperbolic.
Enlargements of geodesics in ${\cal E\cal G}$
are uniform quasi-geodesics in ${\cal G}$.
\end{corollary}
\begin{proof} For any pair $(x,y)$ of vertices in ${\cal G}$ let
$\eta_{x,y}$ be a reparametrization on $[0,1]$ of 
the path $\rho_{x,y}$. 
By Proposition 3.5 of \cite{H07}, 
it suffices to 
show that there is some $n>0$ such that
the paths $\eta_{x,y}$ have the following
properties (where $d$ 
is the distance in ${\cal G}$).
\begin{enumerate}
\item If $d(x,y)\leq 1$ then the diameter of 
$\eta_{x,y}[0,1]$ is at most $n$.
\item For $x,y$ and $0\leq s<t<1$ the 
Hausdorff distance between
$\eta_{x,y}[s,t]$ and 
$\eta_{\eta_{x,y}(s),\eta_{x,y}(t)}[0,1]$ is 
at most $n$.
\item For all vertices $x,y,z$ the set
$\eta_{x,y}[0,1]$ is contained in the $n$-neighborhood
of $\eta_{x,y}[0,1]
\cup \eta_{y,z}[0,1]$.
\end{enumerate}  
Properties 1) and 2) above are immediate from 
Lemma \ref{Lipschitz}. The thin triangle
property 3) follows from Proposition \ref{thin}.
\end{proof}

The following corollary is an
immediate consequence of Corollary \ref{secondlevel}.

\begin{corollary}\label{quasiconvex}
There is a number $k>0$ such that each of the subgraphs
$H_c$ $(c\in {\cal C})$ is $k$-quasi-convex.
\end{corollary}

We complete this section with a calculation of the Gromov 
boundary of ${\cal G}$.

Let as before ${\cal E\cal G}$ be the ${\cal H}$-electrification
of ${\cal G}$. Denote by $\partial {\cal E\cal G}$ be the
Gromov boundary of ${\cal E\cal G}$. For each 
$c\in {\cal C}$ let moreover $\partial H_c$ be the Gromov boundary of 
$H_c$. We equip 
\[\partial {\cal G}=\partial {\cal E\cal G}\cup_c\partial H_c\]
with a topology which is determined by describing for each 
point $\xi\in \partial {\cal G}$ a neighborhood basis as follows.

Let first $\xi\in \partial {\cal E\cal G}$. 
Let $L>1$ be such that every point 
$x\in {\cal G}$ can be connected in ${\cal E\cal G}$ 
to every point $\zeta\in 
\partial{\cal E\cal G}$ by an $L$-quasi-geodesic.

Let $\delta_{\cal E}$ be the Gromov 
metric on $\partial {\cal E\cal G}$ based at a fixed point $x\in {\cal G}$.
For $\epsilon >0$ let ${\cal C}_\xi(\epsilon)$ be the collection of 
all $c\in {\cal C}$ such that there is an $L$-quasi-geodesic 
$\gamma$ connecting $x$ to a point in the $\epsilon$-neighborhood of 
$\xi$ in $(\partial {\cal E\cal G},\delta_{\cal E})$ 
whose tail $\gamma[-\log \epsilon,\infty)$ passes through
the $p(L)$-neighborhood of $v_c$. 
Define 
$B_\epsilon(\xi)\subset \partial {\cal G}$ by
\[B_\epsilon(\xi)=\{\zeta\subset \partial {\cal E\cal G},
\delta_{\cal E}(\zeta,\xi)<\epsilon\}\cup \bigcup_{c\in {\cal C}_\xi(\epsilon)}
\partial H_c.\]
Clearly we have $\cap_{\epsilon >0}B_\epsilon(\xi)=\xi$.
Moreover, changing the basepoint $x$ yields an equivalent neighborhood 
basis.

If $c\in {\cal C}$ and $\xi\in \partial H_c$ 
then choose a basepoint $x\in H_c$ and an 
$L$-quasi-geodesic $\gamma:[0,\infty)\to H_c$
connecting $\gamma(0)=x$ to $\xi$. For $\epsilon >0$ 
let ${\cal C}_\xi(\epsilon)$
be the collection of all $d\in {\cal C}$
such that a geodesic in ${\cal E\cal G}$ connecting
$x$ to $v_d$ passes through the $p(L)$-neighborhood of 
$\gamma[-\log\epsilon,\infty)$ in $H_c$. 
Note that this makes sense since the vertex $v_c$ is 
only connected to vertices in $H_c$ by an edge.

Let moreover $\hat B_\epsilon$ be the set of all 
$\beta\in \partial {\cal E\cal G}$ such that 
an $L$-quasi-geodesic in ${\cal E\cal G}$ 
connecting $x$ to $\beta$ passes through the
$p(L)$-neighborhood of $\gamma[-\log \epsilon,\infty)$ in 
$H_c$. Define 
\[B_\epsilon(\xi)=\hat B_\epsilon\cup \bigcup_{d\in C_\xi(\epsilon)}\partial H_d.\]
By the bounded penetration property, this makes sense. Declare the 
family of sets $B_\epsilon(\xi)$ to be a neighborhood basis of 
$\xi\in \partial \zeta$.
We have

\begin{proposition}\label{boundary}
$\partial {\cal G}$ is the Gromov boundary of ${\cal G}$.
\end{proposition}
\begin{proof} For a number $L>1$ define an
\emph{unparametrized $L$-quasi-geodesic} in 
the graph ${\cal E\cal G}$ 
to be a path $\eta:[0,\infty)\to 
{\cal E\cal G}$ with the following property. 
There is some $n\in (0,\infty]$, and there is an increasing
homeomorphism  $\rho:[0,n)\to [0,\infty)$ such that
$\eta\circ \rho$ is an $L$-quasi-geodesic in ${\cal E\cal G}$.

Let $x\in \zeta$ and 
let $p>1$ be sufficiently large that $x$ can be connected to every
point in the Gromov boundary of ${\cal G}$
by a $p$-quasi-geodesic ray in ${\cal G}$. 
Let $\gamma:[0,\infty)\to {\cal G}$ 
be such a simplicial $p$-quasi-geodesic ray. We claim that 
there is a number $p^\prime >1$ such that $\gamma$ viewed
as a path in ${\cal E\cal G}$ is an unparametrized 
$p^\prime$-quasi-geodesic
in ${\cal E\cal G}$.

Namely, for each $i>0$ let $\zeta_i$ be an 
enlargement of a 
geodesic in  
${\cal E\cal G}$ with endpoints $\gamma(0),\gamma(i)$.
Then there is a number $L>1$ such that $\zeta_i$ is an 
$L$-quasi-geodesic in ${\cal G}$.
By hyperbolicity,
the Hausdorff distance in 
${\cal G}$ between $\gamma[0,i]$ and 
the image of $\zeta_i$ is uniformly bounded, Then  
the same holds true if this Hausdorff distance is measured 
with respect to the distance in 
${\cal E\cal G}\supset {\cal G}$. 
Thus the Hausdorff distance in ${\cal E\cal G}$ between
$\gamma[0,i]$ and a geodesic with the same endpoints is uniformly bounded.
Since $i>0$ was arbitrary, 
this implies that $\gamma$ is an unparametrized
$p^\prime$-quasi-geodesic in ${\cal E\cal G}$ for a number
$p^\prime>0$ only depending on $p$.

As a consequence, if the diameter of $\gamma[0,\infty)$ in 
${\cal E\cal G}$ is infinite then 
up to para\-met\-ri\-za\-tion, $\gamma[0,\infty)$ is
a $p^\prime$-quasi-geodesic ray in ${\cal E\cal G}$
and hence it 
converges as $i\to \infty$
to a point 
$\xi\in \partial {\cal G}$.

Now assume that the diameter of $\gamma[0,\infty)$ in 
${\cal E\cal G}$ is finite.
By Corollary \ref{quasiconvex}, there is a number
$M>0$ not depending on $c$ with the following properties.
\begin{enumerate}
\item If $x,y\in {\cal G}$ are any two vertices and
if $c\in {\cal C}$ is such that
the distance in $H_c$ of some shortest distance
projection of $x,y$ into $H_c$ is at least $M$ then 
a geodesic connecting $x$ to $y$ in ${\cal E\cal G}$
passes through the special vertex $v_c$ defined by $c$.
\item Let $\gamma:[0,\ell]\to {\cal G}$ be a
$p$-quasi-geodesic. If there is some $k\leq \ell$ and some 
$c\in {\cal C}$ such that the
distance in $H_c$ of some shortest distance projection of 
$\gamma(0),\gamma(k)$ into $H_c$ is at least $2M$ then
for each $\ell\geq k$  the distance in $H_c$ 
of a shortest distance projection of 
$\gamma(0),\gamma(\ell)$ into $H_c$ is at least $M$.
\end{enumerate}

For $k>0$ let ${\cal C}_1(k)$ (or ${\cal C}_2(k)$)
be the set of all $c\in {\cal C}$ so that 
the distance in $H_c$ between a shortest distance projection
of $\gamma(0),\gamma(k)$ into $H_c$ is at least
$M$ (or $2M$). By property (2) above, 
for $\ell\geq k$ we have
${\cal C}_2(\ell)\subset {\cal C}_1(k)$.  

The diameter of the image of any simplicial 
geodesic in ${\cal E\cal G}$ equals the length
of the geodesic and hence it is bounded
from below by the number of special vertices it passes through.
Since the diameter of $\gamma[0,\infty)$ in 
${\cal E\cal G}$ is finite by assumption, by property (1) 
the cardinality of ${\cal C}_1(k)$ is bounded 
independent of $k$. 

By property (2) above, we deduce that 
there is some $k_0>0$ so that
$C_2(\ell)\subset C_1(k_0)$ for all $\ell\geq k_0$.
Since the diameter of $\gamma[k_0,\infty)$ in 
${\cal E\cal G}$
is finite, it follows that there is some 
$c\in {\cal C}$ so that  
$\gamma[k_0,\infty)$ is contained in a uniformly bounded
neighborhood of $H_c$. Now $\gamma$ is a $p$-quasi-geodesic 
in ${\cal E\cal G}$ and the embedding 
$H_c\to{\cal E\cal G}$ is a quasi-isometry.
Thus by hyperbolicity, there is a quasi-geodesic 
$\zeta$ in $H_c$ whose Hausdorff distance to 
$\gamma[k_0,\infty)$ is bounded. From
Corollary \ref{secondlevel} we conclude
that $\gamma$ converges as $j\to \infty$ 
to some $\mu\in \partial H_c$.

To summarize, 
there is a map $\Lambda$ from the Gromov boundary of ${\cal G}$
into $\partial {\cal G}$. Corollary \ref{secondlevel} and 
the above discussion shows that $\Lambda$ is a bijection.
The description of the topology on the boundary of ${\cal G}$
as the topology described above for 
$\partial {\cal G}$ 
is now immediate from the description of the Gromov boundary
of a hyperbolic metric graph.
\end{proof}

\section{Thick subsurfaces }

In this section we consider a handlebody 
$H$ of genus $g\geq 2$. 
By a \emph{disc} in $H$
we mean an essential disc in $H$. 

Two discs $D_1,D_2\subset H$ are in 
\emph{normal position} if their boundary circles
intersect in the minimal number of points and if 
every component of $D_1\cap
D_2$ is an embedded arc
in $D_1\cap D_2$ 
with endpoints in $\partial D_1
\cap \partial D_2$.
In the sequel we always assume that discs are in normal position;
this can be achieved by modifying one of the two discs with an
isotopy.

As in the introduction,
call a connected essential subsurface $X$ of $\partial H$ 
\emph{thick} if the following conditions are satisfied.
\begin{enumerate}
\item Every disc intersects $X$. 
\item $X$ is filled by boundaries of discs.
\end{enumerate}
An example of a thick
subsurface is the complement in $\partial H$ of a 
suitably chosen simple
closed curve which is not discbounding. The 
entire boundary surface $\partial H$ is thick as well.

For a thick subsurface $X$ of $\partial H$  
define ${\cal E\cal D\cal G}(X)$ to be the
graph whose vertices are discs with boundary contained in  $X$.
By property (1) in the definition of a thick subsurface,
the boundary of each such vertex is an essential simple closed
curve in $X$. Two such discs $D,E$ are connected by an edge of length one
if and only if 
there is an essential
simple closed curve $\gamma$ in $X$ which 
can be realized disjointly from both $D,E$
(e.g. the boundary of $D$ if the discs $D,E$ are disjoint).

Denote by $d_{{\cal E},X}$ 
the distance in ${\cal E\cal D\cal G}(X)$.
The disc graph ${\cal D\cal G}(X)$ of $X$ is defined in the
obvious way, and we denote by 
$d_{{\cal D},X}$ its distance function.

In the sequel we always assume that all curves and multicurves
on $X\subset \partial H$ are essential. 
For two simple closed multicurves $c,d$ on $\partial H$ let 
$\iota(c,d)$ be the geometric 
intersection number between $c,d$. 

The following lemma \cite{MM04} implies that 
for every thick subsurface $X$ of $\partial H$ the
graph ${\cal D\cal G}(X)$ is connected.
For its proof and later use, let $D,E$ be discs in minimal position.
Define an \emph{outer component}
of $E$ with respect to $D$ 
to be a component $\hat E$ of $E-D$ which is a disc whose boundary
consists of a single subarc of $\partial E$ and a single subarc
$\alpha$ of $D$. The arc $\alpha$ intersects the boundary of
$D$ precisely at its endpoints.
Surgery of $D$ at this outer component $\hat E$ 
replaces $D$ by the union of $\hat E$ with one of the 
two components
of $D-\alpha$ (compare e.g. \cite{MM04,H11}).

\begin{lemma}\label{fellowtravel}
Let $X\subset \partial H$ be a thick subsurface. 
Let $D,E\subset H$ be 
discs with boundaries in $X$. Then $D$
can be connected to a disc $E^\prime$ which is disjoint from 
$E$ by at most $\iota(\partial D,\partial E)/2$ simple surgeries.
In particular, 
\[d_{{\cal D},X}(D,E)\leq \iota(\partial D,\partial E)/2+1.\]
\end{lemma}
\begin{proof} Let $D,E$ be two discs in normal position
with boundary in $X$.
Assume that 
$D,E$ are not disjoint. Then 
there is an outer component of $E-D$. 
A disc $D^\prime$  obtained by simple surgery
of $D$ at this component is essential in $\partial H$ 
and its boundary is contained in $X$, i.e. 
$D^\prime\in {\cal E\cal D\cal G}(X)$. 
Moreover, 
$D^\prime$ 
is disjoint from $D$, i.e. we have 
$d_{{\cal D},X}(D^\prime,D)=1$,  and 
\begin{equation}\label{intersection}  
\iota(\partial E, \partial D^\prime)\leq 
\iota(\partial D,\partial E)-2.\end{equation}
The lemma now follows by induction on 
$\iota(\partial D,\partial E)$.
\end{proof}

\begin{remark}\label{fourhole}
Lemma \ref{fellowtravel} implies that
a thick subsurface
$X$ of $\partial H$  
can not be a four-holed sphere or a one-holed torus.
Namely, if $X$ is a four-holed sphere or a one-holed
torus and if 
$X$ contains the boundaries of two 
distinct discs $D,E$ then these discs
intersect. Surgery of $D$ at an
outer component of
$E-D$ results in an essential disc $D^\prime$ 
which up to homotopy is disjoint from the disc $D$ and 
whose boundary is contained
in $X$.  
Since any two essential simple closed curves
in $X$ intersect, 
the boundary of $D^\prime$ is peripheral in $X$
which violates the assumption that no boundary component
of $X$ is discbounding.
\end{remark}

A simple closed multicurve $\gamma$ in 
a thick subsurface $X$ of 
$\partial H$ is called
\emph{discbusting} if $\gamma$ intersects 
every disc with boundary in $X$.

Consider an oriented $I$-bundle 
${\cal J}(F)$ over a compact (not 
necessarily orientable) surface $F$
with (not necessarily connected) boundary $\partial F$.
The boundary $\partial {\cal J}(F)$ 
of ${\cal J}(F)$ 
decomposes into 
the \emph{horizontal boundary} 
and the \emph{vertical boundary}. 
The  vertical boundary is the interior of 
the restriction of the $I$-bundle
to $\partial F$ and consists of a collection of 
pairwise disjoint open incompressible annuli.
The horizontal boundary is the complement of 
the vertical boundary in $\partial {\cal J}(F)$.

For a given boundary component $\alpha$ of $F$,
the union of the horizontal boundary of
${\cal J}(F)$ 
with the $I$-bundle over $\alpha$
is a compact connected orientable surface 
$F_\alpha\subset \partial {\cal J}(F)$. 
The boundary of $F_\alpha$ is empty if and only if the 
boundary of $F$ is connected.
If the boundary of $F$ is not connected
then $F_\alpha$ is properly contained in the
boundary $\partial {\cal J}(F)$ of 
${\cal J}(F)$. The complement 
$\partial {\cal J}(F)-F_\alpha$ is a union of 
incompressible annuli.

\begin{definition}\label{discbustingibundle}
An \emph{$I$-bundle generator} in 
a thick subsurface $X\subset \partial H$
is an essential simple
closed curve $\gamma\subset X$
with the following property.
There is a compact surface $F$ with non-empty boundary
$\partial F$, there is a boundary component $\alpha$ of 
$\partial F$, and there is an orientation preserving
embedding $\Psi$ of the oriented $I$-bundle
${\cal J}(F)$ over $F$ into $H$ which 
maps $\alpha$ to $\gamma$ and which maps $F_\alpha$
onto the complement in $X$ of a tubular neighborhood 
of the boundary $\partial X$ of $X$.
\end{definition}

We call the surface $F$ the \emph{base} of the $I$-bundle generated by $\gamma$.

\begin{example}
{\bf 1)} An orientable $I$-bundle over an orientable base is 
a trivial bundle. Thus if 
$\partial H$ admits an $I$-bundle 
generator $\gamma$ with orientable base surface $F$ then the genus $g$ of 
$\partial H$ is even and equals twice the genus of 
$F$. 
The $I$-bundle over every essential arc in $F$
with endpoints in $\partial F$ is an embedded disc in $H$.
There is an orientation reversing involution
$\Phi:H\to H$ whose fixed point set intersects $\partial H$
precisely in $\gamma$.  This involution acts as
a reflection in the fiber. The union of any essential
arc $\beta$ in $F$ with endpoints in $\partial F$ with its image 
under $\Phi$ is the boundary of a disc in $H$ (there is a small 
abuse of notation here since the fixed point set of $\Phi$
intersects $\partial H$ in a subset of the fibre over $\partial F$).
This disc is just the $I$-bundle over the arc $\beta$.\\
{\bf 2)}
Let $F$ be an oriented surface of 
genus $k\geq 1$ 
with two boundary components. The oriented $I$-bundle 
${\cal J}(F)=F\times [0,1]$ over $F$
is homeomorphic to a handlebody $H$ of genus $2k+1$. 
A boundary component  
$\beta$ of $F$ is neither discbounding nor 
discbusting in $H$,
and the subsurface $X=\partial H-\beta\subset
\partial H$ is thick. The second boundary 
component $\gamma$ of $F$ 
intersects every disc with boundary in $X$
and is an 
$I$-bundle generator for 
$X$ whose base is the surface $F$.
The image of $F\times [0,1]={\cal J}(F)$ under the 
embedding ${\cal J}(F)\to H$ is the complement of a 
neighborhood of $\beta$ in $H$ which is homeomorphic
to a solid torus.\\
{\bf 3)} Let $F$ be the connected sum of $g$ copies of 
the real projective plane with a disc.
The orientable $I$-bundle 
over $F$ is a handlebody $H$ of genus $g$. 
The vertical boundary of the $I$-bundle is 
an annulus whose core curve $\gamma$ is non-separating. 
The complement of the annulus is the two-sheeted orientation
cover of $F$.
The $I$-bundle over any simple arc in $F$ with both endpoints
on the boundary of $F$ is an embedded disc in $H$.\\
{\bf 4)} Let $\gamma$ be a non-separating $I$-bundle generator for
a proper thick subsurface $X$ of $\partial H$, with 
base $F$. Then $F$ is non-orientable. 
Up to isotopy, the thick subsurface $X$ of 
$\partial H$ is the intersection of 
the boundary $\partial {\cal J}(F)$ 
of the bundle ${\cal J}(F)\subset H$ 
with $\partial H$. It can be obtained
from the orientation cover $\hat F$ of $F$ by
glueing an annulus to the two 
preimages of the preferred boundary component  $\alpha$ of $F$
with a homeomorphism which reverses the boundary orientations. 
The $I$-bundle over every essential arc $\beta$ in $F$ with 
endpoints on $\alpha$ is a disc in $H$.
Its boundary is the preimage of $\beta$ in 
$F_\alpha\subset \partial {\cal J}(F)$, viewed as the orientation
cover of $F$
(here we use the same small abuse of 
terminology as before).
\end{example}

For a thick subsurface $X$ of $\partial H$ 
let ${\cal S\cal D\cal G}(X)$ be the graph
whose vertices are discs with boundaries 
contained in $X$ 
and where two such discs $D,E$ are connected by an edge
of length one if one of the following two
possibilities is satisfied.
\begin{enumerate}
\item There is an
essential simple closed curve $\alpha\subset  X$ 
(i.e. which is essential as a curve in the 
subsurface $X$
of $\partial H$) which is disjoint from $D\cup E$
(for example, $\partial D$ if $D,E$ are disjoint).
\item There is an $I$-bundle generator $\gamma\subset X$
which intersects
both $D,E$ in precisely two points.
\end{enumerate}
We denote by $d_{\cal S,X}$ the distance in ${\cal S\cal D\cal G}(X)$.
If $X=\partial H$ then we simply write $d_{\cal S}$ instead
of $d_{{\cal S},\partial H}$.

The following was proved in \cite{H11} in the
case $X=\partial H$. The proof of the result carries
over to an arbitrary thick subsurface without modification.

\begin{proposition}\label{distanceinter2}
Let $X\subset \partial H$ be a thick subsurface.
The vertex inclusion defines a quasi-isometric
embedding of ${\cal S\cal D\cal G}(X)$ into 
the curve graph of $X$. In particular, 
${\cal S\cal D\cal G}(X)$ is $\delta$-hyperbolic
for a number $\delta >0$ only depending on the genus of $H$. 
\end{proposition}

\section{Hyperbolicity of 
the electrified disc graph}

As in Section 3, we consider a handlebody $H$ of genus 
$g\geq 2$, with boundary $\partial H$.  
The goal of this section is to use Theorem \ref{hypback} to show 
hyperbolicity of the electrified disc graph 
${\cal E\cal D\cal G}(X)$ of a thick subsurface $X$
of $\partial H$. 
We also determine the Gromov boundary
of ${\cal E\cal D\cal G}(X)$.

Thus let $X\subset \partial H$ be a thick subsurface. 
Recall that $X$ is connected, and 
by the remark after Lemma \ref{fellowtravel}, $X$ is 
distinct from a sphere with at most four holes and
from a torus with a single hole. 
Denote by
$d_{{\cal C\cal G},X}$ the distance in the curve graph 
${\cal C\cal G}$ of $X$,
by $d_{{\cal S},X}$ the distance in the graph
${\cal S\cal D\cal G}(X)$
and by $d_{{\cal E},X}$ the distance in the electrified disc graph
${\cal E\cal D\cal G}(X)$ of $X$.

If $X$ does not contain any $I$-bundle generator then
${\cal E\cal D\cal G}(X)={\cal S\cal D\cal G}(X)$ and there is nothing
to show. Thus assume that there is an $I$-bundle generator 
$\gamma\subset X$. Let ${\cal E}(\gamma)\subset
{\cal E\cal D\cal G}(X)$ be the complete subgraph of 
${\cal E\cal D\cal G}(X)$ 
whose
vertices are discs intersecting $\gamma$ in precisely two points.
Define  
${\cal E}=\{{\cal E}(\gamma)\mid \gamma\}$ where 
$\gamma$ runs through all $I$-bundle generators in $X$.
By definition, ${\cal S\cal D\cal G}(X)$ is $2$-quasi-isometric to 
the ${\cal E}$-electrification of ${\cal E\cal D\cal G}(X)$.
Thus by Theorem \ref{hypextension}, 
to show hyperbolicity of ${\cal E\cal D\cal G}(X)$ it
suffices to show that each of the
graphs ${\cal E}(\gamma)$ is $\delta$-hyperbolic for a universal
number $\delta >0$ and that the bounded penetration property
holds true.

We begin with establishing hyperbolicity of the graphs
${\cal E}(\gamma)$.
To this end, for a compact (not necessarily orientable)
surface $F$ with boundary $\partial F$ 
and for a fixed boundary component $\alpha$ of $F$, define  
the \emph{electrified arc graph} $C^\prime(F,\alpha)$ as follows.
Vertices of $C^\prime(F,\alpha)$ are essential 
embedded arcs in $F$ with both
endpoints in $\alpha$. Two such arcs are connected by an edge of length
one if either they are disjoint or if they are disjoint from a common
essential simple closed curve.
If $F$ is non-orientable, then we require that
an essential simple closed curve does not bound a Moebius band in $F$.

The following statement is well known but hard
to find in the literature. We give a proof for completeness.

\begin{lemma}\label{model}
Let $F$ be a compact surface with boundary $\partial F$. Assume that
$F$ is not a sphere with 
at most three holes or a projective plane with at most three holes.
Let $\alpha$ be a boundary circle of $F$. 
Then $C^\prime(F,\alpha)$ 
is $4$-quasi-isometric to the curve graph of $F$.
\end{lemma}
\begin{proof} Define the \emph{arc and curve graph} ${\cal A}(F,\alpha)$ 
of $F$ to be the graph whose vertices are arcs with endpoints on 
$\alpha$ or essential simple closed curves in $F$.
Two such arcs or curves are connected by an edge
of length one if they can be realized disjointly.

Consider first the case that $F$ either is a one-holed torus,
a one-holed Klein bottle, a 
four holed sphere or 
a four-holed projective plane. 
In this case two simple closed curves in $F$ are connected
by an edge in the curve graph of $F$ if they intersect
in the minimal number of points (one or two). Let $\beta$
be an essential simple closed curve in $F$. 
Cutting $F$ open along $\beta$ yields a three-holed sphere
(if $F$ is a one-holed torus or a one-holed Klein bottle), 
the disjoint union of two three holed spheres (if $F$ is a four-holed sphere)
or the disjoint 
union of a three holed sphere and a three holed projective plane
(if $F$ is a four-holed projective plane). 

Thus there is a unique essential 
arc $\Lambda(\beta)\subset F$ with endpoints on $\alpha$
which is disjoint from $\beta$.
The distance between two essential simple closed curves 
$\beta,\gamma$ in the curve graph of $F$
equals one if and only if the arcs 
$\Lambda(\beta),\Lambda(\gamma)$ are disjoint.
This means that the map $\Lambda$ which 
associates to a simple closed
curve $\beta$ in $F$ the unique arc $\Lambda(\beta)$
with endpoints on $\alpha$  
which is disjoint from $\beta$ defines an isometry of 
the curve graph of $F$ onto 
the \emph{arc graph} of $(F,\alpha)$. This arc graph is
the complete subgraph of ${\cal A}(F,\alpha)$ whose vertex set consists of 
arcs with endpoints on $\alpha$. 
Moreover, in the special case at hand, this arc graph is just the
graph $C^\prime(F,\alpha)$. This yields the statement of the lemma
for one-holed tori, one-holed projective planes,
four holed spheres and four-holed 
projective planes.

Now assume that the surface $F$ is such that two 
vertices in the curve graph of $F$ are connected
by an edge if they can be realized disjointly.
Then for any two disjoint
essential simple closed curves $\beta,\gamma$ in $F$
there is an essential arc  
with endpoints on $\alpha$ which
is disjoint from both $\beta,\gamma$. 
In particular, for every 
simplicial path $c$ in 
the arc and curve graph ${\cal A}(F,\alpha)$ connecting
two vertices in ${\cal A}(F,\alpha)$ which are arcs 
with endpoints on $\alpha$, 
there is a 
path of at most double length in $C^\prime(F,\alpha)$ 
connecting the same endpoints.
This path can be obtained from $c$ as follows.
If $c(i),c(i+1)$ are both simple closed curves
then replace $c[i,i+1]$ 
by a simplicial path in ${\cal A}(F,\alpha)$ of length 2 
with the same endpoints whose midpoint is 
an arc disjoint from $c(i),c(i+1)$.
In the resulting path, a simple closed curve $\beta\subset F$ is adjacent to
two arcs disjoint from $\beta$ and hence we can view
this path as a path in $C^\prime(F,\alpha)$.
Thus the vertex inclusion $C^\prime(F,\alpha)\to 
{\cal A}(F,\alpha)$ is a quasi-isometry.  

We are left with showing that ${\cal A}(F,\rho)$ is quasi-isometric
to the curve graph of $F$. 
However, this is well known, and the proof will be omitted.
\end{proof}

A thick subsurface $X$ of $\partial H$ is not
a four-holed sphere. 
Thus if $\gamma$ is a separating 
$I$-bundle generator for $X$ then the base of the 
$I$-bundle either has positive genus or is a sphere with
at least four holes. Similarly, 
if $\gamma$ is a non-separating
$I$-bundle generator for $X$ then we may assume that 
the base $F$ of the $I$-bundle is not a projective plane 
with three holes.
Namely, if $F$ is a projective plane with three holes
and if $\alpha$ is  a distinguished boundary component of $F$ then 
there is up to homotopy 
a unique essential arc $\beta$ in $F$ with boundary
on $\alpha$.
The $I$-bundle over $\beta$ is the unique disc in the oriented 
$I$-bundle over $F$ 
which intersects the curve $\alpha$
in precisely two points.

We use Lemma \ref{model}
to verify hyperbolicity of the subgraphs ${\cal E}(\gamma)$.
For the formulation of the following lemma, for an $I$-bundle
generator $\gamma$ in a thick subsurface $X$ of $\partial H$,
with base surface $F$, denote again by $\gamma$ the distinguished
boundary component of $F$. 
A disc $D\subset H$ 
with boundary $\partial D\subset X$ which intersects $\gamma$ in precisely two 
points is an $I$-bundle over a simple arc 
$\beta \subset F$ with boundary on $\gamma$.
We call $\beta$ the \emph{projection} of $\partial D$ to $F$. 
With these notations we show.

\begin{lemma}\label{arcanddisc}
Let $X\subset \partial H$ be a thick
subsurface and let 
$\gamma$ be an $I$-bundle generator in $X$, 
with base surface $F$. Then the map which 
associates to a disc $D\in {\cal E}(\gamma)$ the projection of 
$\partial D$ to $F$ extends to a  
$2$-quasi-isometry of ${\cal E}(\gamma)$ onto 
the electrified arc graph 
${\cal C}^\prime(F,\gamma)$ of $F$.
\end{lemma}
\begin{proof} Let $\gamma$ be an $I$-bundle generator in $X$ and let 
$F$ be the base surface of the 
$I$-bundle generated by $\gamma$. Let $V$ be the oriented
$I$-bundle over $F$ as in the definition of an 
$I$-bundle generator 
and let $\Psi:V\to H$ be a corresponding embedding.
Up to isotopy, we have 
$\Psi(\partial V)\cap \partial H=X$.
 There is an orientation reversing bundle involution $\Phi$ of $V$
 which exchanges the endpoints of the fibres. The 
 involution preserves $X\subset \partial V$
 and the curve $\gamma$.  The quotient of $X$ under this involution 
 equals the base surface $F$ of the
 $I$-bundle. The projection of $\gamma$ is the distinguished boundary
 component of $F$, again denoted by $\gamma$.

Up to isotopy, if the boundary $\partial D$ of a disc 
$D$ in $H$  
is contained in $X$ and 
intersects the curve $\gamma$ in precisely two points 
then $\partial D$ is invariant under
the involution $\Phi$. Thus the map $\Theta:
{\cal V}(C^\prime(F,\gamma))\to {\cal V}({\cal E}(\gamma))$
which associates to an arc $\beta$ in $F$ with endpoints
on $\gamma$ the $I$-bundle over $\beta$ is a bijection.
Here ${\cal V}(C^\prime(F,\gamma))$ (or 
${\cal V}({\cal E}(\gamma))$) is the set of vertices
of $C^\prime(F,\gamma)$ (or ${\cal E}(\gamma)$).

If $\alpha,\beta\in {\cal V}(C^\prime(F,\gamma))$ are connected by an
edge then either $\alpha,\beta$ are disjoint and so are
$\Theta(\alpha),\Theta(\beta)$, or $\alpha,\beta$ are disjoint from an
essential simple closed curve $\rho$ in  $F$ 
and therefore the discs $\Theta(\alpha),
\Theta(\beta)$ are disjoint from $\rho\subset X$. 
Thus $\Theta$ extends to a 1-Lipschitz map
$C^\prime(F,\gamma)\to {\cal E}(\gamma)$.

We are left with showing that $\Theta^{-1}:{\cal V}({\cal E}(\gamma))\to 
{\cal V}(C^\prime(F,\gamma))$ is $2$-Lipschitz where 
${\cal V}({\cal C}(\gamma))$ and 
${\cal V}(C^\prime(F,\gamma))$ are equipped with 
the restriction of the metric on ${\cal C}(\gamma),C^\prime(F,\gamma)$.
To this end let $\alpha,\beta\in {\cal V}(C^\prime(F,\gamma))$ be
such that $\Theta(\alpha),\Theta(\beta)$ are 
connected by an edge in ${\cal E}(\gamma)$. If $\Theta(\alpha),\Theta(\beta)$
are disjoint then the same holds true for $\alpha,\beta$ and 
hence $\alpha,\beta$ are connected by an edge in 
$C^\prime(F,\gamma)$. Otherwise $\Theta(\alpha),\Theta(\beta)$ are
disjoint from an essential simple closed curve $\rho$.

The boundaries $\partial \Theta(\alpha), \partial \Theta(\beta)$ 
of the discs $\Theta(\alpha),\Theta(\beta)$ are 
invariant under 
the involution $\Phi$ and therefore $\partial \Theta(\alpha)
\cup \partial \Theta(\beta)$ is disjoint 
from $\rho\cup \Phi(\rho)$. As a consequence,
the projection of  $\rho$ to the base surface $F$ 
is  a union of essential arcs with boundary on $\gamma$
and closed curves (not necessarily simple) which are disjoint from
$\alpha\cup \beta$. Then either there is a simple arc in $F$ 
with endpoints on $\gamma$ which is disjoint from 
$\alpha\cup \beta$, or 
there is an essential simple closed curve 
in $F$ which 
is disjoint from $\alpha\cup \beta$. 
Thus
the distance in ${\cal C}^\prime(F,\gamma)$ between
$\alpha\cup\beta$ is at most two.
The lemma follows.
\end{proof}

From Lemma \ref{arcanddisc}, Lemma \ref{model} and hyperbolicity
of the curve graph of $X$ (\cite{MM99}, and \cite{BF07} 
for the curve graph of a non-orientable surface) we
immediately obtain

\begin{corollary}\label{hypfamily}
There is a number $\delta>0$ such that each of the
graphs ${\cal E}(\gamma)$ is $\delta$-hyperbolic.
\end{corollary}

Note that the number $\delta>0$ in the statement of the 
corollary only depends on $H$ but not on $X$.
In fact, the main result of \cite{HPW13} together
with Lemma \ref{arcanddisc} shows that 
it can even be chosen independent of $H$.

We are
left with the verification of the bounded penetration 
property. To this end 
recall from \cite{MM00} the definition
of a \emph{subsurface projection}. Namely, 
let again $X\subset \partial H$ be a thick subsurface and 
let $Y\subset X$ be an essential, 
open connected subsurface which is distinct from 
$X$, a three-holed sphere and
an annulus. We call such a subsurface $Y$ a \emph{proper}
subsurface of $X$.
The arc and curve graph 
${\cal A\cal C}(Y)$ of $Y$ (here we do not specify a 
boundary component) is the graph
whose vertices are isotopy classes of 
arcs with endpoints on $\partial Y$ or essential 
simple closed curves in $Y$,
and two such vertices are connected by an edge of length
one if they can be realized disjointly. 
The vertex inclusion of the curve graph of $Y$ into the
arc and curve graph is a quasi-isometry \cite{MM00}.

There is a projection $\pi_Y$ of the curve graph
${\cal C\cal G}(X)$ of $X$ into 
the space of subsets of  
${\cal A\cal C}(Y)$ which associates to 
a simple closed curve in $X$ the homotopy classes of its
intersection components with $Y$.
For every simple closed multicurve $c$, the diameter of $\pi_Y(c)$ 
in ${\cal A\cal C}(Y)$ is at most one.
If $c$ can be realized disjointly from $Y$ then 
$\pi_Y(c)=\emptyset$.

As before, 
call a path $\rho$ in a metric
graph $G$ simplicial if $c$ maps each
interval $[k,k+1]$ (where $k\in \mathbb{Z}$) isometrically 
onto an edge of $G$.
The following lemma is a version of Theorem 3.1 of \cite{MM00}.

\begin{lemma}\label{projectionlarge}
For every number $L>1$ there is a number $\xi(L)>0$ with the
following property.
Let $Y$ be a proper subsurface of $X$ and let 
$\rho$ be a simplicial path in ${\cal C\cal G}(X)$ which is an 
$L$-quasi-geodesic. If
$\pi_Y(v)\not=\emptyset$ for every vertex $v$ on $\rho$ then
\[{\rm diam}\,\pi_Y(\rho)< \xi(L).\]
Moreover, $\xi(L)\leq ML^3+M$ for a universal constant $M>0$.
\end{lemma}
\begin{proof} By hyperbolicity, for every $L>1$ there is a
number $n(L)>0$ so that for 
every $L$-quasi-geodesic $\rho$ 
in ${\cal C\cal G}(X)$ of finite length, the Hausdorff distance 
between the image of $\rho$ and the image of 
a geodesic $\rho^\prime$
with the same endpoints does not exceed $n(L)$.
Indeed, there is a number $k>0$ only 
depending on the hyperbolicity constant for ${\cal C\cal G}(X)$
such that we can choose $n(L)=kL^2$ (Proposition III.H.1.7 in 
\cite{BH99}).

Now let $Y\subset X$ be a proper subsurface.
By Theorem 3.1 of \cite{MM00}, there is a number $M>0$ 
with the following property. If $\zeta$ is any simplicial 
geodesic in 
${\cal C\cal G}(X)$ and if 
$\pi_Y(\zeta(s))\not=\emptyset$
for all $s\in \mathbb{Z}$ in the domain of $\zeta$ then 
\[{\rm diam}(\pi_Y(\zeta))\leq M.\]
 
Let $L>1$, let 
$\rho:[0,k]\to {\cal C\cal G}(X)$ be a simplicial path
which is an $L$-quasi-geodesic
and assume that 
\begin{equation}\label{dia}
{\rm diam}(\pi_Y(\rho(0)\cup \rho(k)))\geq 2M+L(2n(L)+4).\notag
\end{equation}
Our goal is to show that $\rho$ passes through the set
$A\subset {\cal C\cal G}(X)$ of all 
essential simple closed
curves in $X-Y$. 
The diameter of $A$ in ${\cal C\cal G}(X)$  
is at most two.

To this end 
let $\rho^\prime$ be a simplicial geodesic in ${\cal C\cal G}(X)$ 
with the same endpoints as $\rho$.
Theorem 3.1 of \cite{MM00} 
shows that there is some 
$u\in \mathbb{Z}$ such that $\rho^\prime(u)\in A$. Then $\rho$
passes through the $n(L)$-neighborhood of $A$.

Let $s+1\leq t-1$ be the smallest and the biggest number,
respectively, so that $\rho(s+1),\rho(t-1)$ are contained 
in the $n(L)$-neighborhood of $A$.
Then $\rho[0,s]$ (or $\rho[t,k]$) is
contained in the complement of the $n(L)$-neighborhood
of $A$. Since $\rho$ is an $L$-quasi-geodesic, 
a geodesic connecting $\rho(0)$ to $\rho(s)$
(or connecting $\rho(t)$ to $\rho(k)$)
is contained in the $n(L)$-neighborhood of $\rho[0,s]$ 
(or of $\rho[t,k]$) and
hence it does not pass through $A$.
In particular, 
\[{\rm diam}(\pi_Y(\rho(0)\cup \rho(s)))\leq M \text{ and }
{\rm diam}(\pi_Y(\rho(t)\cup \rho(k)))\leq M.\]
As a consequence, we have
\begin{equation}\label{diam}
{\rm diam}(\pi_Y(\rho(s)\cup \rho(t)))\geq L(2n(L)+4).
\end{equation}

Since $d_{{\cal C\cal G},X}(\rho(s+1),A)\leq n(L)$ and
$d_{{\cal C\cal G},X}(\rho(t-1),A)\leq n(L)$ and 
since the diameter of $A$ is at most $2$,  
we obtain
$d_{{\cal C\cal G},X}(\rho(s),\rho(t))\leq 2n(L)+2$.
Now $\rho$ is  a simplicial $L$-quasi-geodesic 
in ${\cal C\cal G}(X)$ and 
hence the length $t-s$ of $\rho[s,t]$ is at most
$L(2n(L)+2)+L=L(2n(L)+3)$. 
For all $\ell\in \mathbb{Z}$ the curves 
$\rho(\ell),\rho(\ell+1)$ are
disjoint and therefore if $\rho(\ell),\rho(\ell+1)$ both intersect
$Y$ then the diameter  of $\pi_Y(\rho(\ell)\cup 
\rho(\ell+1))$ is at most one. Thus
if $\rho(\ell)$ intersects $Y$ for all $\ell$ then  
\[{\rm diam}(\pi_Y(\rho(s)\cup \rho(t)))\leq L(2n(L)+3).\]
This contradicts inequality (\ref{diam}) and completes the proof of the
lemma.
\end{proof}

For simplicity of notation, in the remainder of this section
we identify discs in $H$ with their boundaries.
In other words, for a thick subsurface $X$ of $\partial H$ we
view the vertex sets of the graphs ${\cal S\cal D\cal G}(X),
{\cal E\cal D\cal G}(X)$ as subsets of the vertex set of the
curve graph ${\cal C\cal G}(X)$ of $X$.

Let ${\cal S\cal D\cal G}_0(X)$ be the ${\cal E}$-electrification 
of ${\cal E\cal D\cal G}(X)$. 
For each $I$-bundle generator $\gamma$ in $X$, the 
graph ${\cal S\cal D\cal G}_0(X)$ contains a special vertex $v_\gamma$.
The vertex set of ${\cal S\cal D\cal G}_0(X)$ is the union of 
the set of all discbounding simple closed curves in $X$ 
with the set $\{v_\gamma\mid \gamma\}$.
In particular, there is a natural vertex inclusion 
${\cal V}({\cal S\cal D\cal G}_0(X))\to {\cal C\cal G}(X)$ which 
maps the special vertex
$v_\gamma$ to the simple closed curve $\gamma$. 
Since ${\cal S\cal D\cal G}(X)$ is quasi-isometric to the
${\cal E}$-electrification of ${\cal E\cal D\cal G}(X)$, 
Proposition \ref{distanceinter2}
shows that this vertex inclusion extends to 
a quasi-isometric embedding ${\cal S\cal D\cal G}_0(X)\to {\cal C\cal G}(X)$. 

Now we are ready to show

\begin{lemma}\label{boundedinelec}
For every thick subsurface $X$ of $\partial H$ 
the family ${\cal E}$ has the bounded penetration property.
\end{lemma}
\begin{proof} Let $L\geq 1$ and let 
$\rho:[0,n]\to {\cal S\cal D\cal G}_0(X)$ 
be an efficient simplicial $L$-quasi-geodesic. 
Let $\tilde\rho$ be a simplicial arc in ${\cal C\cal G}(X)$ which 
is obtained from $\rho$ 
as follows.

A vertex $\rho(j)$ in 
${\cal S\cal D\cal G}_0(X)$ which 
is not one of the special vertices $v_\gamma$ 
also defines a vertex in ${\cal C\cal G}(X)$.
If $\rho(j),\rho(j+1)$ are two such vertices which 
are connected in ${\cal S\cal D\cal G}_0(X)$ by an edge
then they are connected in ${\cal E\cal D\cal G}(X)\subset
{\cal S\cal D\cal G}_0(X)$ by an edge. By the definition of the
electrified disc graph, this means that there is a simple
closed curve $\alpha$ in $X$ which is disjoint from $\rho(j)\cup \rho(j+1)$.
Thus $\rho(j)$ and  $\rho(j+1)$ 
can be connected in 
${\cal C\cal G}(X)$ by an edge path of length at most two.

Similarly, if $\rho(j)=v_\gamma$ for an $I$-bundle
generator $\gamma$ in $X$, then $\rho(j-1),\rho(j+1)\in 
{\cal E\cal D\cal G}(X)$, moreover $\rho(j-1),\rho(j+1)$ 
intersect $\gamma$ in precisely two points. Replace 
$\rho[j-1,j+1]$ by an edge path in 
${\cal C\cal G}(X)$ with the same endpoints of length at
most four which passes through $\gamma$. 
The arc $\tilde\rho$ constructed in this way from $\rho$ is a uniform
quasi-geodesic in ${\cal C\cal G}(X)$ which 
passes through any $I$-bundle generator $\gamma$ at most once,
and it passes through $\gamma$ if and only if it 
passes through a simple closed curve which is
disjoint from $\gamma$.

Let $\gamma$ be a separating $I$-bundle generator in $X$.
Then $X-\gamma$ has two homeomorphic components $X_1,X_2$.
Denote by $d_{{\cal A\cal C},X_i}$ the distance in the arc
and curve graph of $X_i$ $(i=1,2)$.
Every simple closed curve $\alpha$ in $X$ which has an essential
intersection with $\gamma$ projects to a collection of arcs
$\alpha_1,\alpha_2$ in $X_1,X_2$. If $\beta$ is another such curve
then define 
\[d_{{\cal A\cal C}(X-\gamma)}(\alpha,\beta)=
\min\{d_{{\cal A\cal C}(X_1)}(\alpha_1,\beta_1),
d_{{\cal A\cal C}(X_2)}(\alpha_2,\beta_2)\}.\]
Thus if $\pi^\gamma:{\cal C\cal G}(X)\to 
{\cal A\cal C}(X-\gamma)={\cal A\cal C}(X_1)\cup {\cal A\cal C}(X_2)$ denotes
the subsurface projection then  
by Lemma \ref{projectionlarge}, there is a number 
$M(L)>0$ with the following property. 

Let again $\rho:[0,n]\to {\cal S\cal D\cal G}_0(X)$ be a simplicial 
$L$-quasi-geodesic. If 
\[d_{{\cal A\cal C}(X-\gamma)}(\pi^\gamma(\rho(0)),\pi^\gamma(\rho(n)))
\geq M(L)\]
then there is some $k_0\in \mathbb{Z}$ such that 
$\tilde\rho(k_0)=\gamma$.
Equivalently, there is some $k<n$ such that
$\rho(k)=v_\gamma$. Moreover, 
\[d_{{\cal A\cal C}(X_i)}(\pi^\gamma(\rho(0)), \pi^\gamma(\rho(k-1)))\leq
M(L)\,(i=1,2),\] and similarly
\[d_{{\cal A\cal C}(X_i)}(\pi^\gamma(\rho(k+1)), 
\pi^\gamma(\rho(n)))\leq M(L)\, (i=1,2).\]

As a consequence, if 
$\rho^\prime:[0,n^\prime]\to {\cal S\cal D\cal G}_0(X)$ 
is another efficient quasi-geodesic with the same endpoints,
then there is some $k^\prime<n^\prime$ 
such that $\rho^\prime(k^\prime)=v_\gamma$, 
and 
\begin{align}
d_{{\cal A\cal C}(X-\gamma)}(\pi^\gamma(\rho(k-1)),
\pi^\gamma(\rho^\prime(k^\prime-1)))\leq 2M(L),\notag\\
d_{{\cal A\cal C}(X-\gamma)}
(\pi^\gamma(\rho(k+1)),\pi^\gamma(\rho^\prime(k^\prime+1)))\leq 2M(L).
\notag\end{align}
Lemma \ref{arcanddisc}
and Lemma \ref{model} now show that the 
distance in ${\cal E}(\gamma)$ between 
$\rho(k-1),\rho^\prime(k^\prime-1)$ and between
$\rho(k+1),\rho^\prime(k^\prime+1)$ is uniformly bounded.
In particular, the bounded penetration property holds true
for the subgraph ${\cal E}(\gamma)$ and for quasi-geodesics
connecting two discs whose boundaries have projections
of large diameter into $X-\gamma$.

On the other hand, if $\rho:[0,n]\to {\cal S\cal D\cal G}_0(X)$ is 
any efficient $L$-quasi-geodesic and if $\rho(k)=v_\gamma$ for some
$I$-bundle generator 
$\gamma$ then using once more Lemma \ref{projectionlarge}, we conclude
that 
\[d_{{\cal A\cal C}(X-\gamma)}
(\pi^\gamma(\rho(0)),\pi^\gamma(\rho(k-1)))\leq M(L).\]
Therefore the reasoning in the previous paragraph shows that
whenever the distance in ${\cal E}(\gamma)$ between
$\rho(k-1),\rho(k+1)$ is sufficiently large then 
\[d_{{\cal A\cal C}(X-\gamma)}(\pi^\gamma(\rho(0)),\pi^\gamma(\rho(n)))\geq 
M(L).\]
In other words, the conclusion in the previous
paragraph holds true, and 
the bounded penetration property for separating
$I$-bundle generators follows.
 
Now assume that $\gamma$ is non-separating. Let 
$\pi^\gamma:{\cal C\cal G}(X)\to {\cal A\cal C}(X-\gamma)$ be the subsurface
projection. Using the notations from the beginning of this proof,
if the distance in ${\cal A\cal C}(X-\gamma)$ 
between $\pi^\gamma(\rho(0))$ and $\pi^\gamma(\rho(n))$ is at least
$M(L)$ then there is some $k$ so that $\rho(k)=v_\gamma$. 
Moreover, we have $\rho(k-1)\in {\cal E}(\gamma),
\rho(k+1)\in {\cal E}(\gamma)$.
As a consequence, the curves $\rho(k-1),\rho(k+1)$ are
invariant under the orientation reversing involution $\phi$ of $X$ which 
preserves $\gamma$ and 
extends to an involution of the $I$-bundle defined by $\gamma$.

Let $F$ be the base of the $I$-bundle 
defined by $\gamma$ and let $\alpha,\beta\in C^\prime(F,\gamma)$ be
the projections of $\rho(k-1),\rho(k+1)$. By Lemma \ref{arcanddisc}, 
the distance in ${\cal E}(\gamma)$ between 
$\rho(k-1),\rho(k+1)$ is 
uniformly equivalent 
to the distance in $C^\prime(F,\gamma)$ between $\alpha,\beta$.
Since $\rho(k-1),\rho(k+1)$ are invariant under
the involution $\phi$, the main result of \cite{RS09} shows that
this distance is also uniformly equivalent to the distance 
between $\pi^\gamma(\rho(k-1))$ and $\pi^\gamma(\rho(k+1))$ 
in ${\cal A\cal C}(X-\gamma)$.

In particular, if $\rho^\prime$ is any other 
efficient $L$-quasi-geodesic in ${\cal S\cal D\cal G}_0(X)$  
with the same endpoints, then there is some $k^\prime$ with 
$\rho(k^\prime)=v_\gamma$, and the distance in ${\cal E}(\gamma)$ between
$\rho(k-1),\rho^\prime(k^\prime-1)$ and between
$\rho(k+1)$ and $\rho^\prime(k^\prime+1)$ is uniformly bounded.
The bounded penetration property follows in this case.

Finally, as in the case of a separating
$I$-bundle generator, 
this argument can be inverted. Together this 
completes the proof of the lemma.
\end{proof}

We can now apply Theorem \ref{hypextension} to conclude

\begin{corollary}\label{elect} 
For every thick subsurface $X$ of 
$\partial H$, the graph 
${\cal E\cal D\cal G}(X)$ is $\delta$-hyperbolic for a number
$\delta >0$ not depending on $X$. 
There is a number $k>0$ such that 
for every $I$-bundle generator $\gamma$ in $X$, 
the subgraph ${\cal E}(\gamma)$
of ${\cal E\cal D\cal  G}(X)$ is $k$-quasi-convex. 
\end{corollary}

In the remainder of this section, we specialize to the
case $X=\partial H$. We begin with establishing a
distance estimate for the electrified disc
graph ${\cal E\cal D\cal G}={\cal E\cal D\cal G}(\partial H)$.

If 
$\gamma$ is an $I$-bundle generator in $\partial H$ then let
$\pi^\gamma$ be the subsurface 
projection of a simple closed curve in $\partial H$ 
into the arc and curve-graph of
$\partial H-\gamma$.

For a subset $A$ of a metric space $Y$ 
and a number $C>0$ define
${\rm diam}(A)_C$ to be the diameter of $A$ if this 
diameter is at least $C$ and let ${\rm diam}(A )_C=0$ otherwise.
The notation $\asymp$ means equality up to a universal
multiplicative constant.

\begin{corollary}\label{distancesecond}
Let $H$ is a handlebody of genus $g\geq 2$.
Then there is a number $C>0$ such that 
\[d_{\cal E}(D,E)\asymp 
d_{\cal C\cal G}(\partial D,\partial E) +
\sum_\gamma {\rm diam}(\pi^\gamma(\partial D\cup\partial E))_C\] 
where $\gamma$ passes through all 
$I$-bundle generators on $\partial H$.
\end{corollary}
\begin{proof} Let ${\cal S\cal D\cal G}_0$ 
be the ${\cal E}$-electrification of ${\cal E\cal D\cal G}$.
For an $I$-bundle generator $\gamma$ in $\partial H$
denote by $v_\gamma$ the special vertex in ${\cal S\cal D\cal G}_0$
defined by $\gamma$.

Let 
$\rho:[0,k]\to {\cal S\cal D\cal G}_0$ 
be a geodesic. By Corollary \ref{secondlevel}
and Corollary \ref{elect}, an enlargement 
$\hat \rho$ of $\rho$ is a uniform quasi-geodesic
in ${\cal E\cal D\cal G}$. Thus it suffices to show that
the length of $\hat \rho$ is uniformly comparable to the
right hand side of the formula in the corollary. 
 
By Proposition \ref{distanceinter2}, there is a number $L>1$
such that a simplicial arc $\tilde\rho$ 
in ${\cal C\cal G}$ constructed from 
$\rho$ as in the proof of Lemma \ref{boundedinelec}
is an $L$-quasi-geodesic in the curve
graph ${\cal C\cal G}$ of $\partial H$.
Lemma \ref{boundedinelec} shows that 
if $\hat \rho$
is an enlargement of $\rho$ then the diameter of the
intersection of $\hat \rho$ with ${\cal E}(c)$ equals
the diameter of $\pi_{\partial H-c}(\gamma(0)\cup \gamma(k))$ up to a universal
multiplicative and additive constant. 
This is what we wanted to show.
\end{proof}

We complete this section with determining the Gromov boundary of the
electrified disc graph of $H$.
To this end let
$H$ be a handlebody of genus $g\geq 2$.
Let ${\cal L}$ be the space of all geodesic
laminations on $\partial H$ equipped with the 
\emph{coarse Hausdorff topology} \cite{H06}.
In this topology, a sequence of laminations
$\lambda_i$ converges to $\lambda$ if every accumulation
point of $(\lambda_i)$ in the usual Hausdorff topology 
for compact subsets of $\partial H$ contains 
$\lambda$ as a sublamination.
Let 
\[{\cal H}\subset {\cal L}\] be the subspace of 
all minimal laminations which fill up $\partial H$, i.e. such 
that complementary components are simply connected,
and which are limits in the coarse Hausdorff topology of 
discbounding simple closed curves.

For an $I$-bundle generator $\gamma$ let 
$\partial{\cal E}(\gamma)\subset {\cal L}$ be the 
set of all 
geodesic laminations which consist of two minimal components
filling up $\partial H-\gamma$ and 
which are limits in 
the coarse Hausdorff topology of boundaries
of discs contained in ${\cal E}(\gamma)$. Each lamination $\mu\in 
{\cal E}(\gamma)$ is invariant under the orientation reversing 
involution $\Phi_\gamma$ of $\partial H$ which fixes $\gamma$ pointwise and
exchanges the 
endpoints of the fibres of the defining $I$-bundle.

Define
\[\partial {\cal E\cal D\cal G}=
\partial{\cal H}
\cup \bigcup_\gamma\partial {\cal E}(\gamma)\subset {\cal L}\]
where the union is over all $I$-bundle generators
$\gamma\subset \partial H$. Then $\partial {\cal E\cal D\cal G}$ is a
${\rm Map}(H)$-space. 

Proposition \ref{boundary} can now be applied to show

\begin{proposition}\label{edgb}
The Gromov boundary of ${\cal E\cal D\cal G}$ can 
naturally be identified with $\partial {\cal E\cal D\cal G}$.
\end{proposition}
\begin{proof}
We show first that
the subspace $\partial {\cal E\cal D\cal G}$ of ${\cal L}$ 
is Hausdorff. 

A point $\lambda\in \partial {\cal E\cal D\cal G}$ 
either is a minimal geodesic
lamination which fills up $\partial H$, or it is a 
geodesic lamination with two minimal components
which fill up $\partial H-\gamma$ for some $I$-bundle generator $\gamma$.
Let $\nu\not=\lambda$ 
be another such lamination. We claim that $\nu$ and $\lambda$
intersect. This means that for some fixed hyperbolic metric on 
$\partial H$, the geodesic representatives of $\nu,\lambda$
intersect transversely.

If either $\nu$ or $\lambda$
fills up $\partial H$ (i.e. if the complementary components
of $\nu,\lambda$ are simply connected) 
then this is obvious. Otherwise $\nu$ fills up
the complement
of an $I$-bundle generator $\gamma$, and $\lambda$
fills up the complement of an $I$-bundle generator $\gamma^\prime$.
Now the simple closed curve 
$\gamma$ is the only minimal geodesic lamination 
which does not intersect $\nu$ and which is
distinct from a component of $\nu$. The lamination
$\lambda$ consists of two minimal components which 
are not simple closed curves and therefore
the geodesic laminations  
$\nu,\lambda$ indeed intersect.

Since $\nu,\lambda\in \partial {\cal E\cal D\cal G}$ intersect,
by the definition of the coarse Hausdorff topology
there are neighborhoods $U$ of $\lambda$, $V$ of $\nu$ 
in ${\cal L}$ so that any
two laminations $\lambda^\prime\in U,\nu^\prime\in V$ 
intersect. In particular,
the neighborhoods $U,V$ are disjoint. This shows that
$\partial {\cal E\cal D\cal G}$ is Hausdorff.

Proposition \ref{boundary} shows that there is a natural
bijection between $\partial {\cal E\cal D\cal G}$ and 
the Gromov boundary of ${\cal E\cal D\cal G}$.
That this bijection is in fact a homeomorphism follows from the
description 
the Gromov boundary of the curve graph of 
$\partial H$ as discussed in \cite{K99,H06} and 
Proposition \ref{boundary}.

To be more precise, let $\gamma$ be a
separating $I$-bundle generator for
$\partial H$. The orientation reversing involution 
$\Phi$ of the $I$-bundle
determined by $\gamma$ restricts to a homeomorphism of 
$\partial H-\gamma$ which exchanges the two components of
$\partial H-\gamma$. By Lemma \ref{model} and Lemma \ref{arcanddisc}, 
the graph ${\cal E}(\gamma)$ can be identfied
with the graph of all simple closed curves $\alpha$ in $X$ which 
intersect $\gamma$ in precisely in two points and are invariant under
$\Phi$. 
Thus 
by \cite{K99,H06}, the Gromov boundary of ${\cal E}(\gamma)$ has a 
natural identification with the space of all $\Phi$-invariant
geodesic laminations which consist of two minimal components, each
of which fills a component of $\partial H-\gamma$. The topology on 
this space is the coarse Hausdorff topology. 
A similar description is valid for the Gromov boundary
of ${\cal E}(\gamma)$ where $\gamma$ is an orientation reversing
$I$-bundle generator. 

Proposition \ref{boundary} shows that the Gromov boundaries of the 
subspaces ${\cal E}(\gamma)$ are embedded subspaces of
the Gromov boundary of 
${\cal E\cal D\cal G}$. The Gromov boundary 
${\cal H}$
of ${\cal S\cal D\cal G}$
is embedded in the Gromov boundary of ${\cal E\cal D\cal G}$ as well. 
For every $\xi\in {\cal H}$, a neighborhood basis of 
$\xi$ in the Gromov boundary of ${\cal E\cal D\cal G}$ consists of 
sets which are unions of a neighborhood of $\xi$ in ${\cal H}$ 
with sets $\partial {\cal E}(\gamma)$ where the curves $\gamma$ are
contained in a neighborhood of $\xi$ in 
${\cal C\cal G}\cup \partial {\cal C\cal G}$. By the description of neighborhood bases
of $\xi$ in ${\cal C\cal G}\cup \partial{\cal C\cal G}$ as neighborhoods of 
$\xi$ in lamination space, equipped with the coarse Hausdorff topology
\cite{K99,H06}, this completes the proof of the proposition.
\end{proof}

\section{Hyperbolicity of the disc graph}

In this section we use Corollary \ref{elect} and 
Theorem \ref{hypback} to give a new and simpler proof of 
the following result of Masur and Schleimer \cite{MS13}.

\begin{theorem}\label{dischyp}
The disc graph of a handlebody is hyperbolic.
\end{theorem}

The argument consists in an inductive application
of Theorem \ref{hypback} to electrified disc graphs of 
thick subsurfaces of $\partial H$. 
For technical reason we slightly weaken the definition
of a thick subsurface of $\partial H$ as follows.

Define a connected properly embedded
subsurface $X$ of $\partial H$ to be \emph{visible} if every
discs intersects $X$ 
and if moreover $X$ contains the boundary of at least one disc.
Thus a thick subsurface is visible, but a visible subsurface
may not be filled by boundaries of discs and hence may not be thick.
Note that if $X$ is visible then the electrified disc
graph ${\cal E\cal D\cal G}(X)$ of $X$ is defined. However, 
if $X$ is not thick then its 
diameter equals one.

Let ${\cal D\cal G}(X)$ be the disc graph of the visible
subsurface $X$. Its vertices are discs with boundary in
$X$, and two such discs are connected by an edge of length one
if they are disjoint.
The next lemma shows that 
if $X$ is a 
visible five holed sphere or two holed torus then 
${\cal D\cal G}(X)$ is hyperbolic.

\begin{lemma}\label{fiveholedsphere}
Let $X\subset \partial H$ be a visible subsurface which
is a five-holed sphere or a two-holed torus.
Then the vertex inclusion ${\cal D\cal G}(X)\to 
{\cal E\cal D\cal G}(X)$ is a quasi-isometry.
\end{lemma}
\begin{proof} Let $X\subset \partial H$ be a visible subsurface. 
Let $\rho:[0,k]\to {\cal E\cal D\cal G}(X)$ be
a geodesic. By modifying $\rho$ while increasing
its length by at most a factor of two we may assume that
for each $i$, either $\rho(i)$ is disjoint from
$\rho(i+1)$, or there is an essential simple closed curve
in $X$ 
which is not discbounding and which is 
disjoint from both $\rho(i),\rho(i+1)$, but there is no
discbounding curve in $X$ 
disjoint from both $\rho(i),\rho(i+1)$.

If $X$ is a five-holed sphere then every simple closed curve
$\gamma$ in $X$ is separating, and $X-\gamma$ is the disjoint 
union of a four holed
sphere $X_1$ and a three holed sphere. Any two essential 
simple closed curves $\alpha,\beta$ 
in $X$ which are disjoint from $\gamma$ are contained in $X_1$.
If $\gamma\subset X$ is not discbounding 
then $X_1$ is a 
four holed sphere whose boundary components
are not discbounding. 
If $\gamma$ is disjoint
from a discbounding simple closed curve then $X_1$ contains the 
boundary of a disc. 
By Remark \ref{fourhole}, $X_1$ contains 
the boundary of precisely one disc.
This implies that for all $i$ the disc
$\rho(i)$ is disjoint from $\rho(i+1)$ and 
therefore $\rho$ is in fact a simplicial path in 
${\cal D\cal G}(X)$.

Similarly, if $X$ is a two-holed 
torus then a simple closed curve $\gamma$
in $X$ either is non-separating and  $X-\gamma$ is a four-holed sphere, 
or $\gamma$ is separating and $\gamma$ decomposes $X$ into a three-holed sphere
and a one-holed torus. Using again Remark \ref{fourhole}, 
a one-holed torus whose boundary is not discbounding contains
the boundary of at most one disc. Thus 
the argument in the previous paragraph is
valid in this situation as well and shows the lemma.  
\end{proof}

From now on let $X$ be a visible subsurface of $\partial H$
which is not a sphere with at most five holes or a 
torus with at most two holes. 
Our next goal is to 
show hyperbolicity of a graph ${\cal E\cal D\cal G}(2,X)$
whose vertices are isotppy classes of essential discs with 
boundary in $X$, and which can be obtained from 
${\cal D\cal G}(X)$ by adding edges and can be 
obtained from ${\cal E\cal D\cal G}(X)$ by removing edges.
Namely, 
two discs $D,E$ are connected in ${\cal E\cal D\cal G}(2)$ 
by an edge of length
one if either $D,E$ are disjoint or if
$\partial D,\partial E$ are disjoint from an essential multicurve
$\beta\subset \partial X$ consisting of two components
which are not freely homotopic.

Call a simple closed curve  $\gamma$
in $X$ \emph{admissible} if  $\gamma$ has the following properties.
\begin{enumerate}
\item  $\gamma$ is neither discbounding nor discbusting. 
\item Either $\gamma$  is non-separating 
or $\gamma$ decomposes $X$ into 
a three-holed sphere $X_1$ and a visible
second component $X_2$. 
\end{enumerate}

By assumption, $X$ is distinct from a sphere with at
most five holes and a torus with at most two holes. We claim that  
if $\gamma\subset X$ is an admissible simple closed curve and 
if $\eta$ is any other simple closed curve then a tubular neighborhood 
of $\gamma\cup \eta$ contains an essential simple closed curve which is
disjoint from $\gamma$. 

To see this observe that if $\gamma\cap \eta=\emptyset$ then we may choose 
$\eta$ to be such a curve. 
If $\gamma\cap \eta\not=\emptyset$ then 
let $\eta_0$ be a component of $\eta-\gamma$. In the case that $\gamma$ 
is separating we require that $\eta_0$  
is not contained in the 
three-holed sphere component of $X-\gamma$.
Then $\eta_0$ is contained in a component of $X-\gamma$ which 
neither is a three holed sphere nor a one-holed torus.
As a consequence, 
one of the boundary components of a tubular neighborhood of 
$\gamma\cup \eta_0$ is an 
essential simple closed curve in $X$ distinct from $\gamma$.

For an admissible simple closed curve $\gamma$ in $X$ 
define ${\cal H}(\gamma)$
to be the complete subgraph of ${\cal E\cal D\cal G}(2,X)$
whose vertex set consists of all discs which are disjoint from $\gamma$.

\begin{lemma}\label{hise}
${\cal H}(\gamma)$ is isometric to ${\cal E\cal D\cal G}(X-\gamma)$.
\end{lemma}
\begin{proof} 
A disc $D\in {\cal H}(\gamma)$ is disjoint from $\gamma$.
Thus $D$ defines a vertex in ${\cal E\cal D\cal G}(X-\gamma)$.
Two discs $D,E\in {\cal H}(\gamma)$ are connected by an edge in 
${\cal E\cal D\cal G}(2,X)$ if and only if either they are disjoint or if 
there is a pair $\beta_1,\beta_2$ of disjoint 
essential simple closed curves in $X$ 
which are disjoint from both $D,E$. 

If one of the curves $\beta_1,\beta_2$, say the curve
$\beta_1$, is disjoint from $\gamma$, then by definition, 
$D,E$ viewed as vertices in ${\cal E\cal D\cal G}(X-\gamma)$ 
are connected by an edge.  Otherwise by the remark
preceding the lemma, there is an essential
simple closed curve contained in a tubular neighborhood of 
$\gamma\cup \beta_1$ which is disjoint from $\gamma,D,E$ 
and once again,
$D,E$ are connected by an edge in ${\cal E\cal D\cal G}(X-\gamma)$.

As a consequence, the vertex inclusion ${\cal H}(\gamma)\to 
{\cal E\cal D\cal G}(X-\gamma)$ extends to a $1$-Lipschitz embedding.
By definition, this embedding is surjective on vertices. By definition, 
any two vertices which are connected in 
${\cal E\cal D\cal G}(X-\gamma)$ by an edge are also connected
in ${\cal H}(c)$ by an edge. In other words,
the $1$-Lipschitz embedding ${\cal H}(\gamma)\to 
{\cal E\cal D\cal G}(X-\gamma)$ is in fact an isometry.
\end{proof}

Lemma \ref{hise} and Corollary \ref{elect} imply that 
there is a number $\delta >0$ so that
each of the graphs ${\cal H}(\gamma)$ is $\delta$-hyperbolic.

Let ${\cal H}=\{{\cal H}(\gamma)\mid \gamma\}$ 
be the family of all these subgraphs
of ${\cal E\cal D\cal G}(2,X)$ where $\gamma$ 
passes through all admissible curves in 
$X$. Our goal is to apply Theorem \ref{hypextension} to the family
${\cal H}$ of complete subgraphs of ${\cal E\cal D\cal G}(2,X)$.
We first note.

\begin{lemma}\label{helect}
${\cal E\cal D\cal G}(X)$ is quasi-isometric to the
${\cal H}$-electrification of ${\cal E\cal D\cal G}(2,X)$.
\end{lemma}
\begin{proof} Let ${\cal G}$ be the ${\cal H}$-electrification of 
${\cal E\cal D\cal G}(2,X)$. We show first that the vertex inclusion
${\cal E\cal D\cal G}(X)\to {\cal G}$ is coarsely Lipschitz.

To this end 
let $D,E$ are any two vertices in ${\cal E\cal D\cal G}(X)$ which
are connected by an edge. Then either $D,E$ are disjoint,
of they are disjoint from a common essential 
simple closed curve $\gamma$ in $X$.

If $D,E$ are disjoint then $D,E$ viewed as vertices
in ${\cal E\cal D\cal G}(2,X)$ 
are connected by an edge in ${\cal E\cal D\cal G}(2,X)$ as well.

Now assume that $D,E$ are disjoint from 
a common essential simple closed curve $\gamma$ in $X$.
If $\gamma$ either is admissible or discbounding, 
then by the definition of the ${\cal H}$-electrification
of ${\cal E\cal D\cal G}(2,X)$, 
their distance in ${\cal G}$ is at most two.
On the other hand, if  $\gamma$ is 
neither admissible nor discbounding, 
then $\gamma$ is a separating simple closed curve in $X$.
The surface $X-\gamma$ is a disjoint union of essential 
surfaces $X_1,X_2$ which are distinct from three-holed spheres.
The boundaries of $D,E$ are contained in $X_1\cup X_2$.

If $\partial D,\partial E$ are contained in distinct
components of $X-\gamma$ then $D,E$ are disjoint and hence
$D,E$ are connected by an edge in ${\cal E\cal D\cal G}(2,X)$.
If $\partial D,\partial E$ are contained 
in the same component of $X-\gamma$, say
in $X_1$, then the second component $X_2$ contains an essential
simple closed curve $\eta$, and 
$\partial D,\partial E$ are disjoint from
the multi-curve $\gamma\cup \eta$ with two components.
Once more, this implies that 
$D,E$ are connected in ${\cal E\cal D\cal G}(2,X)$ by an edge.
As a consequence, the vertex inclusion 
${\cal E\cal D\cal G}(X)\to {\cal G}$ is indeed coarsely Lipschitz.

That this inclusion is in fact a quasi-isometry is immediate from the 
definitions. Namely, if $\gamma\subset X$ 
is admissible then by the definition of
${\cal E\cal D\cal G}(X)$, any two vertices in 
${\cal H}(\gamma)$ are connected in 
${\cal E\cal D\cal G}(X)$ by an edge. 
\end{proof}

We use Lemma \ref{helect} and Theorem \ref{hypextension} 
to show hyperbolicity 
of ${\cal E\cal D\cal G}(2,X)$.

\begin{corollary}\label{secondlevel2}
${\cal E\cal D\cal G}(2,X)$ is hyperbolic.
Enlargements of geodesics in ${\cal E\cal D\cal G}(X)$
are uniform quasi-geodesics in ${\cal E\cal D\cal G}(2,X)$.
\end{corollary}
\begin{proof} 
It suffices to show that
the family ${\cal H}=\{H(\gamma)\mid \gamma\}$ 
satisfies the assumptions in 
Theorem \ref{hypextension}.

For an admissible simple closed curve $\gamma\subset X$,
$\delta$-hyperbolicity of ${\cal H}(\gamma)$ 
for a number $\delta >0$ not depending on $\gamma$ 
follows from 
Corollary \ref{hise} and Corollary \ref{elect}.

To show the bounded penetration property, recall that
enlargements of geodesics in ${\cal S\cal D\cal G}(X)$ are
uniform quasi-geodesics in ${\cal E\cal D\cal G}(X)$. 
Let $\gamma$ be an admissible simple closed curve and let $X_1$
be the component of $X-\gamma$ which is not a three-holed sphere. By 
Lemma \ref{projectionlarge},
Lemma \ref{boundedinelec}  and the proof of Corollary \ref{elect},
an enlargement of a geodesic in ${\cal S\cal D\cal G}(X)$ 
passes through two points 
of large distance in ${\cal H}(\gamma)={\cal E\cal D\cal G}(X-\gamma)$ if and
only if one of the following two possibilities holds true.
\begin{enumerate}
\item The diameter of the 
subsurface projection of the endpoints into the 
arc and curve graph of $X_1$
is large.
\item There is an $I$-bundle generator $\beta\subset X_1$ 
such that the diameter of the subsurface projection of the endpoints
into the arc and curve graph of $X_1-\beta$ is large.
\end{enumerate} 
From this the bounded penetration property follows as in the
proof of Lemma \ref{boundedinelec}.
\end{proof}

\bigskip

{\it Proof of Theorem \ref{dischyp}:} 
For $k\geq 1$ define
${\cal E\cal D\cal G}(k)$ to be the graph whose
vertex set is the set of all discs in $H$ and where
two such discs are connected by an edge of length one if and
only if either they are disjoint or they are both disjoint from 
a multicurve in $\partial H$ with a least $k$ components.
Note that if $k$ equals the cardinality of 
a pants decomposition for $\partial H$ then ${\cal E\cal D\cal G}(k)$
equals the disc graph of $H$.

We show by induction on $k$ 
the following.
The graph ${\cal E\cal D\cal G}(k)$
is hyperbolic, and there is a collection ${\cal H}$ of 
complete hyperbolic subgraphs of ${\cal E\cal D\cal G}(k)$ which
satisfies the hypothesis in Theorem \ref{hypextension}
and such that the ${\cal H}$-electrification of 
${\cal E\cal D\cal G}(k)$ is naturally quasi-isometric to 
${\cal E\cal D\cal G}(k-1)$. In particular,  
enlargements of quasi-geodesics
in ${\cal E\cal D\cal G}(k-1)$ are uniform quasi-geodesics in 
${\cal E\cal D\cal G}(k)$.

The case $k=1$ is just Corollary \ref{elect}, and
the case $k=2$ was shown in Corollary \ref{secondlevel2}.
Thus assume that the claim holds true for $k-1\in [2,3g-3)$.
Let $D,E$ be any two discs in $H$. 
Let $\rho$ be a geodesic in 
${\cal E\cal D\cal G}(k-1)$ connecting $D$ to $E$.

Let $i\geq 0$ be such that the discs 
$\rho(i),\rho(i+1)$ are not disjoint. Then they
are  disjoint from a multicurve $\alpha$ in $\partial H$ 
with at least $k-1$
components. We may assume that none of the components of 
$\alpha$ is discbounding.
Since $\partial \rho(i),\partial \rho(i+1)$ intersect
they are contained in the same component 
$X$ of $\partial H-\alpha$. Then $X$ is a visible
subsurface of $\partial H$.

If either $\partial H-X$ contains a multicurve
with $k$ components or if $X-(\partial \rho(i)\cup \partial \rho(i+1))$ 
contains an essential simple closed curve 
then $\rho(i)$ is connected to
$\rho(i+1)$ by an edge
in ${\cal E\cal D\cal G}(k)$. 
Otherwise replace the edge 
$\rho[i,i+1]$ in ${\cal E\cal D\cal G}(k-1)$ 
by a geodesic $\rho_k^i$ 
in ${\cal E\cal D\cal G}(X)$ 
with the same endpoints. The concatenation of these arcs is a
curve $\rho_k$ which 
is an enlargement of $\gamma$. For all $j$,
the discs $\gamma_k(j),\gamma_k(j+1)$ either are disjoint
or disjoint from a multicurve containing at least $k$
components which are not discbounding.

This process stops in the moment the component $X$ of 
$\partial H-\alpha$ is a five-holed sphere or a two-holed torus
since by Lemma \ref{fiveholedsphere},  for 
such a surface the vertex inclusion 
${\cal D\cal G}(X)\to {\cal E\cal D\cal G}(X)$ is a quasi-isometry.
\qed

\bigskip

For a thick subsurface $Y$ of $\partial H$ 
denote as before 
by $\pi_Y$ the subsurface projection of simple closed curves 
into the arc and curve graph of $Y$. 
If $\gamma$ is an $I$-bundle generator
in a thick subsurface $Y$ then let $\pi^\gamma$ be the subsurface projection
into $Y-\gamma$.
 
The following corollary is now immediate from our construction.
It was earlier obtained by Masur and Schleimer 
(Theorem 19.9 of \cite{MS13}).

\begin{corollary}\label{distancesecond2}
There is a number $C>0$ such that 
\[d_{\cal D}(D,E)\asymp \sum_Y 
{\rm diam}(\pi_{Y}(E\cup D))_C+\sum_\gamma{\rm diam}(\pi^\gamma(E\cup D))_C\]
where $Y$ passes through all thick subsurfaces of $\partial H$, where
$\gamma$ passes
through all $I$-bundle generators in thick subsurfaces of $\partial H$,
and the diameter is taken in the arc and curve graph.
\end{corollary}

For a thick subsurface $Y$ of $\partial H$ let 
$\partial {\cal E\cal D\cal G}(Y)$ 
be the Gromov boundary of ${\cal E\cal D\cal G}(Y)$.
Define \[\partial {\cal D\cal G}=
\cup_Y\partial{\cal E\cal D\cal G}(Y)\subset {\cal L}\] where
the union is over all thick subsurfaces of $\partial H$
and where this union is viewed as a subspace
of ${\cal L}$, i.e. it is equipped with the coarse
Hausdorff topology. The proof of the following statement
is completely analogous to the proof of Proposition \ref{edgb}
and will be omitted.

\begin{corollary}\label{gromovdisc}
$\partial {\cal D\cal G}$ is the Gromov boundary of 
${\cal D\cal G}$.
\end{corollary}

\bigskip\bigskip

\noindent
MATH. INSTITUT DER UNIVERSIT\"AT BONN\\
ENDENICHER ALLEE 60\\
53115 BONN\\
GERMANY\\

\bigskip\noindent
e-mail: ursula@math.uni-bonn.de

\end{document}